\documentclass[12pt]{amsart}
\usepackage{amsmath,amssymb,amsfonts,epsfig}

\newtheorem{theorem}{Theorem}[section]
\newtheorem{proposition}[theorem]{Proposition}
\newtheorem{lemma}[theorem]{Lemma}
\newtheorem{corollary}[theorem]{Corollary}
\newtheorem{conjecture}[theorem]{Conjecture}

\theoremstyle{definition}
\newtheorem{definition}[theorem]{Definition}
\newtheorem{example}[theorem]{Example}
\newtheorem{problem}[theorem]{Problem}

\theoremstyle{remark}
\newtheorem{remark}[theorem]{Remark}

\numberwithin{equation}{section}

\def\DJ{{\hbox{D\kern-.8em\raise.15ex\hbox{--}\kern.35em}}}
\def\DJo{\DJ okovi\'c}
\def\DZD{D.\v{Z}. \DJo}
\def\pd{{\partial}}

\def\bC{{\bf C}}

\def\bR{{\bf R}}
\def\bH{{\bf H}}
\def\bZ{{\bf Z}}
\def\pA{{\mathcal A}}

\def\pF{{\mathcal F}}
\def\pH{{\mathcal H}}

\def\pJ{{\mathcal J}}
\def\pK{{\mathcal K}}
\def\pO{{\mathcal O}}
\def\pP{{\mathcal P}}

\def\pU{{\mathcal U}}

\def\pX{{\mathcal X}}

\def\diag{{\rm diag}}

\def\sgn{{\rm sgn}}

\def\GL{{\mbox{\rm GL}}}

\def\SO{{\mbox{\rm SO}}}

\def\Sp{{\mbox{\rm Sp}}}
\def\Un{{\mbox{\rm U}}}
\def\SU{{\mbox{\rm SU}}}
\def\CT{{\mbox{\rm CT}}}
\def\la{{\langle}}
\def\ra{{\rangle}}

\def\NE{{\rm NE}}
\def\NW{{\rm NW}}
\def\SW{{\rm SW}}
\def\SE{{\rm SE}}

\def\sgn{{\rm sgn}}
\def\tr{{\,\rm tr\,}}

\def\de{\delta}

\def\vf{\varphi}

\def\sig{\sigma}

\font\germ=eufm10

\def\t{{\mbox{\germ t}}}
\def\u{{\mbox{\germ u}}}

\begin{document}

\title[Zero patterns and unitary similarity]
{Zero patterns and unitary similarity}

\author[J. An]{Jinpeng An$^{1,2}$}

\author[D.\v{Z}. \DJo{}]{Dragomir \v{Z}. \DJo{}$^1$}
\thanks
{The second author was supported in part by an NSERC Discovery Grant.}

\address{1. Department of Pure Mathematics, University of Waterloo,
Waterloo, Ontario, N2L 3G1, Canada}

\address{2. LMAM, School of Mathematical Sciences, Peking University, Beijing, 100871, China}

\email{anjinpeng@gmail.com}
\email{djokovic@uwaterloo.ca}

\begin{abstract}
A subspace of the space, $L(n)$, of traceless complex $n\times n$ matrices
can be specified by requiring that the entries at some positions $(i,j)$
be zero. The set, $I$, of these positions is a (zero) pattern and the
corresponding subspace of $L(n)$ is denoted by $L_I(n)$. A pattern $I$
is universal if every matrix in $L(n)$ is unitarily similar to some matrix
in $L_I(n)$. The problem of describing the universal patterns is raised,
solved in full for $n\le3$, and partial results obtained for $n=4$.
Two infinite families of universal patterns are constructed.
They give two analogues of Schur's triangularization theorem.
\end{abstract}

\date{}



\maketitle

\section{Introduction}

This paper is a sequel to our paper \cite{AD} where we studied the
universal subspaces $V$ for the representation of a connected compact Lie
group $G$ on a finite-dimensional real vector space $U$. The meaning
of the word ``universal'' in this context is that every $G$-orbit
in $U$ meets the subspace $V$. The general results obtained in
that paper have been applied in particular to the conjugation actions
$A\to XAX^{-1}$, $X\in G$, of the classical compact Lie groups $G$,
i.e., $\Un(n)$, $\SO(n)$ and $\Sp(n)$, on
the space of $n\times n$ matrices $M(n,\bC)$, $M(n,\bR)$ and
$M(n,\bH)$, respectively. (By $\bH$ we denote the algebra of real
quaternions.)

In the present paper we restrict our scope to the
complex case, i.e., to $M(n)=M(n,\bC)$ and $G=\Un(n)$.
However, all results where we establish the nonsingularity (see Section
\ref{Nonsing-Obl} for the definition) of certain patterns are directly
applicable to the real and quaternionic cases.
Throughout the paper we denote by $L(n)\subseteq M(n)$ the subspace of
traceless matrices, and by $T_n\subseteq\Un(n)$ the maximal
torus consisting of the diagonal matrices. We shall
consider only a very special class of complex subspaces of $M(n)$;
those that can be specified by requiring
that the matrix entries in specified positions $(i,j)$ vanish.
We denote the set of these positions $(i,j)$ by $I$ and
denote  by $M_I(n)$ the corresponding subspace. We also set
$L_I(n)=L(n) \cap M_I(n)$. We refer to
$I$ as a (zero) pattern and denote the set of all such $I$'s by $\pP_n$.
A pattern is strict if it contains no diagonal positions.
It is proper if it does not contain all the diagonal positions.
We say that a pattern $I\in\pP_n$ is universal if the subspace $L_I(n)$
is universal in $L(n)$. We point out that, for a strict pattern
$I\in\pP_n$, $L_I(n)$ is universal in $L(n)$
iff $M_I(n)$ is universal in $M(n)$.

The main question we consider, the universality problem, is to
determine all universal patterns in $\pP_n$. In full generality,
this problem is solved only for $n\le3$.
There is a simple necessary condition for universality of a
proper pattern $I$: $|I|\le\mu_n=n(n-1)/2$ (see Proposition
\ref{RestrDim} below). We denote by $\pP'_n$ the set of strict
patterns $I\in\pP_n$ with $|I|=\mu_n$.
Theorem 5.1 of \cite{AD} provides a sufficient condition for
the universality of a pattern (see Theorem \ref{TeorAD} below).
We use this result to construct some infinite families of strict
universal patterns. The main results in this direction are
Theorems \ref{SchurAnalog} and \ref{Fam-J}.

In Section \ref{UnivObl} we define the universal patterns and state the
universality problem for patterns $I\in\pP_n$. The case $n=2$ is easy:
All patterns $I\in\pP_2$ of size 1 are universal. The nonsingularity
of all $I\in\pP'_3$ has been established in \cite{AD}. In Proposition
\ref{3-Obl} we show that none of the proper nonstrict patterns $I\in\pP_3$ of
size 3 is universal. Thus the universality problem is solved for $n\le3$.
A proper pattern $I\in\pP_n$ is $n$-defective if the stabilizer of
$L_I(n)$ in $\Un(n)$ has dimension larger than $n^2-2|I|$. Such patterns
are not universal.

In Section \ref{Nonsing-Obl} we introduce the nonsingular patterns.
We say that $I\in\pP_n$ is $n$-nonsingular if $\chi_I\notin\pK(n)$,
where $\chi_I\in\bR[x_1,\ldots,x_n]$ is the product of all differences
$x_i-x_j$, $(i,j)\in I$, and $\pK(n)$ is the ideal generated
by the nontrivial elementary symmetric functions of the $x_i$'s. The
basic fact, that nonsingular patterns are universal, was proved in
\cite{AD}. We say that a pattern $I$ is simple if $(i,j)\in I$
implies that $(j,i)\notin I$. All simple patterns are nonsingular,
and so universal. A pattern $I\in\pP'_n$ is $n$-exceptional if it is
$n$-singular but not $n$-defective. There is no general method for
deciding whether an exceptional pattern $I\in\pP'_n$ is universal.
Proposition \ref{FundPol} provides a simple method for testing
whether a pattern $I\in\pP'_n$ is nonsingular. The inner product
$\la\cdot,\cdot\ra$ used in the proposition is defined in the
beginning of the section.

In Section \ref{Ekv} we introduce two equivalence relations ``$\approx$''
and ``$\sim$'' in $\pP'_n$. We refer to the former simply as
``equivalence'' and to the latter as ``weak equivalence''.
This is justified since $I\approx I'$ implies $I\sim I'$.
If $I\approx I'$ then $I$ is universal iff $I'$ is universal, but
we do not know if this also holds for weak equivalence.
However, if $I\sim I'$ then $I$ is nonsingular iff $I'$ is nonsingular.
For any pattern $I$ we define its complexity $\nu(I)$ as the number
of positions $(i,j)$ with $i\le j$ such that both $(i,j)$ and $(j,i)$
belong to $I$. The patterns of complexity 0 are exactly the simple patterns.
We show that for $n\ge4$ the set of patterns of complexity 1 in
$\pP'_n$ splits into two weak equivalence classes. One of these
classes is singular and the other nonsingular.

In Section \ref{Analog} we consider a particular sequence of
nonsingular patterns $\Lambda_n\in\pP'_n$, $n\ge1$, of maximal complexity,
i.e., with $\nu(\Lambda_n)=[\mu_n/2]$. For convenience let us write
$n=4m+r$ where $m,r\ge0$ are integers and $r<4$. The pattern
$\Lambda_n$ consists of all positions $(i,j)$ with $i\ne j$ and
$i+j\le n+1$, except those of the form $(2i-1,n-2i+1)$ and
$(n-2i+1,2i-1)$ with $1\le i\le m$, and if $r=2$ or 3 we also
omit the position $(2m+2,2m+1)$. As $\Lambda_n$ is nonsingular,
it is also universal, i.e., every matrix $A\in M(n)$ is unitarily
similar to one in the subspace $M_{\Lambda_n}(n)$. Thus we can
view this result as an analogue of Schur's theorem. The whole
section is dedicated to the proof of this result.

In Section \ref{Analogue} we consider an infinite family of patterns
$J(\sig,{\bf i})$ depending on an integer $n\ge1$, a permutation
$\sig\in S_n$ and a sequence ${\bf i}=(i_1,i_2,\ldots,i_{n-1})$ of
distinct integers. This sequence has to be chosen so that,
for each $k$, $|i_k|\in\{\sig(1),\sig(2),\ldots,\sig(k)\}$.
The pattern $J(\sig,{\bf i})$ consists of all positions
$(i_k,\sig(j))$ with $i_k>0$ and $(\sig(j),-i_k)$ with $i_k<0$,
where in both cases $1\le k<j\le n$. The main result of this
section (Theorem \ref{Fam-J}) shows that all the patterns
$J(\sig,{\bf i})$ are nonsingular. As a special case, we obtain
another analogue of Schur's theorem (see Proposition \ref{Primer-2}).

In Section \ref{Cetiri} we consider the exceptional patterns $I\in\pP'_4$.
Up to equivalence, there are seven of them (see Table 2). We prove
that the first two of them are not universal while the third one is.
This is the unique example that we have of a strict pattern which is
singular and universal. For the remaining four patterns in Table 2
the universality question remains open. The same question for the
nonstrict patterns in $\pP_4$ remains wide open.

There are other interesting questions that one can raise about
the subspaces $L_I(n)$ and the unitary orbits
$\pO_A=\{UAU^{-1}:\, U\in\Un(n) \}$, $A\in L(n)$.
For instance, if $L_I(n)$ is not universal
we can ask for the characterization of the set $\Un(n)\cdot L_I(n)$.
Another question of interest is to determine the number, $N_{A,I}$,
of $T_n$-orbits contained in the intersection $\pX_A=L_I(n)\cap\pO_A$.
For instance, if $L_I(n)$ is the space of upper (or lower) triangular
traceless matrices and $A\in L(n)$ has $n$ distinct eigenvalues
then $N_{A,I}=n!$.

In Section \ref{RedFlag} we consider a pattern $I\in\pP'_n$ and a matrix
$A\in L(n)$ such that $\pX_A\ne\emptyset$. The homogeneous space
$\pF_n=\Un(n)/T_n$ is known as the flag manifold and we refer to
its points as flags. If $g^{-1}Ag\in L_I(n)$,
$g\in\Un(n)$, we say that the flag $gT_n$ reduces $A$ to $L_I(n)$.
We say that $A$ is generic if $\pX_A$ and $L_I(n)$ intersect
transversally (see the next section for the definition).
For generic $A$, we show that $N_{A,I}$ is equal to the number of
flags which reduce $A$ to $L_I(n)$.

In Section \ref{CikObl} we consider the case $n=3$ and the cyclic pattern
$I=\{(1,3),(2,1),(3,2)\}$. We know that $I$ is
nonsingular and so $\Un(3)\cdot L_I(3)=L(3)$. We show that the set
$\Theta$ of nongeneric matrices $A\in L(3)$ is contained in
a hypersurface $\Gamma$ defined by $P=0$, where $P$ is a homogeneous
$\Un(3)$-invariant polynomial of degree 24. This polynomial is
explicitly computed and we show that it is absolutely irreducible.
The restriction, $P_I$, of $P$ to $L_I$ factorizes as
$P_I=P_1^2P_2$, where $P_1$ and $P_2$ are absolutely irreducible
homogeneous polynomials of degree 6 and 12, respectively. Thus,
$\Gamma \cap L_I=\Gamma_1 \cup \Gamma_2$ where $\Gamma_i\subset L_I$
is the hypersurface defined by $P_i=0$, $i=1,2$. The hypersurface
$\Gamma_1$ consists of all matrices $A\in L(3)$ such that $\pO_A$
and $L_I$ meet non-transversally at $A$.

We propose a conjecture (see Section \ref{Nonsing-Obl}) and
several open problems.

For any positive integer $n$ we set $\bZ_n=\{1,2,\ldots,n\}$ and
$\mu_n=n(n-1)/2$.
If $\pA$ is a $\bZ$-graded algebra, we denote by $\pA_d$ the
homogeneous component of $\pA$ of degree $d$. We use the same
notation for the homogeneous ideals of $\pA$.

We thank the referee for his suggestions regarding the
presentation and some minor corrections.

\section{Universal patterns} \label{UnivObl}

Let $M(n)$ denote the algebra of complex $n\times n$ matrices
and $L(n)$ its subspace of matrices of trace 0.
We are interested in subspaces of $M(n)$ or $L(n)$ which can be
specified by zero patterns. For that purpose we introduce
the notion of patterns.

A {\em position} is an ordered pair $(i,j)$ of positive integers.
We say that a position $(i,j)$ is {\em diagonal} if $i=j$.
A {\em pattern} is a finite set of positions.
A pattern is {\em strict} if it has no diagonal positions.
The {\em size} of a pattern $I$ is its cardinality, $|I|$.
If $I,I'$ are patterns and $I\subseteq I'$ then we say
that $I'$ is an {\em extension} of $I$.
We denote by $\pP$ the set of all patterns and by $\pP_n$ the set
of patterns contained in $\bZ_n\times\bZ_n$.
We denote by $\pP'_n$ the subset of $\pP_n$ consisting
of the strict patterns of size $\mu_n$.
The ``$n\times n$ zero patterns'' used in our paper \cite{AD}
are the same as the patterns in $\pP'_n$.

We define an involutory map $T:\pP\to\pP$, called {\em transposition},
by setting $I^T=\{(j,i):\, (i,j)\in I\}$ for $I\in\pP$.
We refer to $I^T$ as the {\em transpose} of $I$. We say that $I$ is
{\em symmetric} if $I^T=I$. The sets
$\pP_n$ and $\pP'_n$ are $T$-invariant for all $n$.
Note that if $I\in\pP_n$ is universal (or nonsingular)
then $I^T$ has the same property.

For $I\in\pP_n$, we denote by $M_I(n)$, or just $M_I$ if $n$ is fixed,
the subspace of $M(n)$ which consists of all matrices $X=[x_{ij}]$
such that $x_{ij}=0$ for all $(i,j)\in I$. We also set
$L_I(n)=L(n) \cap M_I(n)$.

Some important patterns in $\pP_n$ are the diagonal pattern
$\Delta_n=\{(i,i):i\in\bZ_n\}$ and the four triangular patterns:
\begin{eqnarray*}
 \NE_n &=& \{(i,j)\in\bZ_n\times\bZ_n: i<j \}, \\
 \SW_n &=& \{(i,j)\in\bZ_n\times\bZ_n: i>j \}, \\
 \NW_n &=& \{(i,j)\in\bZ_n\times\bZ_n: i+j<n+1 \}, \\
 \SE_n &=& \{(i,j)\in\bZ_n\times\bZ_n: i+j>n+1 \}.
\end{eqnarray*}
The first is the {\em upper triangular} and the second
the {\em lower triangular} pattern. Note that, according to this terminology,
if $I$ is the upper triangular pattern, then $M_I$ is the space of lower
triangular matrices.

The unitary group $\Un(n)$ acts on $L(n)$ by conjugation, i.e.,
unitary similarities.

\begin{definition} We say that a real subspace $V\subseteq L(n)$
is {\em universal} in $L(n)$ if every matrix in $L(n)$ is unitarily
similar to a matrix in $V$. We also say that a pattern $I\in\pP_n$
is {\em $n$-universal} if the subspace $L_I(n)$ is universal
in $L(n)$.
\end{definition}

The prefix ``$n$-'' will be supressed if $n$ is clear from
the context. This convention shall apply to several other
definitions that we will introduce later.

It is obvious that a strict $n$-universal pattern
is also $m$-universal for all $m>n$.
The converse is not valid, e.g., the pattern $\{(1,2),(2,1)\}$ is
3-universal but not $2$-universal.

We are interested in the (pattern) universality problem, i.e.,
the problem of deciding which patterns $I\in\pP_n$ are universal.
It is easy to see that, for a strict pattern $I\in\pP_n$, the
subspace $L_I(n)$ is universal in $L(n)$ iff the subspace
$M_I(n)$ is universal in $M(n)$. The Schur's triangularization
theorem asserts that the triangular patterns $\NE_n$ and
$\SW_n$ are $n$-universal.

It is well known that $\Delta_n$ is $n$-universal \cite[Theorem 1.3.4]{HJ}
for all $n$. However, if $i,j\in\bZ_n$ and $i\ne j$ then the next example
implies that the pattern $\Delta_n \cup \{(i,j)\}$ is not universal.

\begin{example} \label{dij+1}
The pattern $I=\{(1,1),(1,2),(2,2)\}$ is not $n$-universal for $n\ge2$.
Let $D=\diag(1,\ldots,1,1-n)\in L(n)$ and let $A\in L_I(n)$. As the rank
of $D-{\rm id}$ is 1 and that of $A-{\rm id}$ is at least 2,
$D$ and $A$ are not similar. Consequently, $L_I$ is not universal.
This implies that $\NW_n$ is not $n$-universal for $n\ge4$. For the case $n=3$
see Proposition \ref{3-Obl}.
\end{example}

We give another example of a nonuniversal nonstrict pattern.
\begin{example} \label{drugi-neuni}
The pattern $J=\{(i,1):1\le i<n\} \cup \{(1,n)\}$ is not $n$-universal for $n\ge2$.
Let $D$ be as above and let $A=[a_{ij}]\in L_I(n)$.
Assume that $D$ and $A$ are similar. Since $D-{\rm id}$ has rank 1,
$A-{\rm id}$ must also have rank 1. Thus
$$
\left| \begin{array}{cc} -1&0\\ a_{n1} & a_{nn}-1 \end{array} \right|=0 \quad
\text{ and } \quad
\left| \begin{array}{cc} -1 & a_{1j} \\ 0 & a_{jj}-1 \end{array} \right|=0
\text{ for } 1<j<n, $$
i.e., $a_{jj}=1$ for $1<j\le n$. As $A\in L_J$, we have $a_{11}=0$ contradicting
the fact that $\tr(A)=0$. Consequently, $L_I$ is not universal.
\end{example}

If $n-1$ of the diagonal entries of a matrix $A\in L(n)$ vanish
then all $n$ of them vanish. For that reason we introduce the
following definition: We say that a pattern $I\in\pP_n$ is
$n$-{\em proper} if $(i,i)\notin I$ for at least one $i\in\bZ_n$.
Observe that, for any pattern $I\in\pP_n$, the subspace $L_I(n)$ is
stabilized by the maximal torus $T_n$. An easy dimension argument
shows that the following is valid, see \cite[Lemma 4.1]{AD}.

\begin{proposition} \label{RestrDim}
Let $I\in\pP_n$ be proper and universal. Then the dimension of the
stabilizer of $L_I(n)$ in $\Un(n)$ does not exceed $n^2-2|I|$.
In particular, $n\le n^2-2|I|$, i.e., $|I|\le\mu_n$.
\end{proposition}

Thus we have a simple condition that any proper universal
pattern $I\in\pP_n$ must satisfy: $|I|\le\mu_n$.

We say that a proper pattern $I\in\pP_n$ is $n$-{\em defective} if
the dimension of the stabilizer of $L_I(n)$ in $\Un(n)$ is larger
than $n^2-2|I|$.
By the proposition, such patterns are not universal. Note that
any proper pattern $I\in\pP_n$ with $|I|>\mu_n$ is defective.

Next, we show that some special extensions of strict universal patterns
are also universal. For $I\in\pP$ and integers $m,n\ge0$ we denote by
$(m,n)+I$ the translate $\{(m+i,n+j):\,(i,j)\in I\}$ of $I$.

\begin{proposition} \label{UnivPros}
Let $I\in\pP_n$ and $J\in\pP_{m-n}$, $m>n$, be strict patterns and
assume that $I$ is $n$-universal and $J$ is $(m-n)$-universal.
Then the pattern
$$ I'=I \cup \left( (n,n)+J \right) \cup
\left( (0,n)+\bZ_n\times\bZ_{m-n} \right) $$
is $m$-universal.
\end{proposition}

\begin{proof}
Given a matrix $A\in L(m)$, choose $X\in\Un(m)$ such that $B=XAX^{-1}$
is lower triangular. Let $B_1$ resp. $B_2$ denote the square submatrix
of $B$ of size $n$ resp. $m-n$ in the upper left resp. lower right
hand corner.
Since $I$ and $J$ are strict and universal, there exist matrices
$Y_1\in\Un(n)$ and $Y_2\in\Un(m-n)$ such that $Y_1B_1Y_1^{-1}\in M_I(n)$
and $Y_2B_2Y_2^{-1}\in M_J(m-n)$.
Thus if $Y=Y_1\oplus Y_2$, then $YBY^{-1}\in L_{I'}(m)$.
Hence $I'$ is $m$-universal.
\end{proof}

Let us fix a positive integer $n$ and a pattern $I\in\pP'_n$.
For $A\in L(n)$ we denote by $\pO_A$ the $\Un(n)$-orbit through $A$,
i.e., $\pO_A=\{UAU^{-1}:\, U\in\Un(n) \}$. The set
$\pX_A=L_I(n)\cap\pO_A$ is closed and $T_n$-invariant.
We say that $\pO_A$ intersects $L_I$ {\em transversally at a point}
$B\in\pX_A$ if the sum of $L_I(n)$ and the tangent space of $\pO_A$
at $B$ is equal to the whole space $L(n)$.
If this is true for all points $B\in\pX_A$, then we say that
$\pO_A$ and $L_I$ intersect {\em transversally},
and that the matrix $A$ and its orbit $\pO_A$ are {\em $I$-generic}.
We shall denote by $N_{A,I}$ the cardinality of the set $\pX_A/T_n$
(the set of $T_n$-orbits in $\pX_A$). We note that $N_{A,I}$
is finite if $A$ is $I$-generic, see Section \ref{RedFlag} and
\cite[Section 4]{AD}.

Almost nothing is known about the universality of the subspaces
$L_I(n)$ of $L(n)$ for nonstrict patterns $I$, but see the above examples.
The case $n=2$ is easy and we leave it to the reader. Let us analyze
the case $n=3$. The case of strict patterns, $\pP'_3$, has been handled
in \cite{AD}. It is easy to see that any pattern $I\in\pP_3$ of size 2 is
universal. By taking into account the above examples and the fact
that $\Delta_3$ is universal, there are only four cases to consider:
\begin{eqnarray*}
&&
\left[ \begin{array}{ccc} 0&0&*\\ 0&*&*\\ *&*&* \end{array} \right],
\quad
\left[ \begin{array}{ccc} 0&*&*\\ *&*&0\\ *&0&* \end{array} \right],
\quad
\left[ \begin{array}{ccc} 0&0&*\\ *&*&*\\ *&0&* \end{array} \right],
\quad
\left[ \begin{array}{ccc} 0&0&*\\ *&*&0\\ *&*&* \end{array} \right].
\end{eqnarray*}
(The starred entries are arbitrary, subject only to the
condition that the matrices must have zero trace.)

We shall prove that none of them is universal and thereby complete
the solution of the universality problem for $n=3$.

\begin{proposition} \label{3-Obl}
For $n=3$, no proper nonstrict pattern of size 3 is universal.
\end{proposition}
\begin{proof}
We need only consider the four subspaces, $V$, mentioned above.
It turns out that in all four cases there exists a diagonal matrix
$D\in L(3)$ such that $\pO_D\cap V=\emptyset$.

In the first two cases the proof consists in constructing an
$\Un(3)$-invariant polynomial function $P:L(3)\to\bR$ which is
nonnegative on $V$ and negative on the diagonal matrices
$D=\diag(u,v,-u-v)$ when $u$ and $v$ are
linearly independent over $\bR$. The polynomial $P$ will be
expressed as a polynomial in the $\Un(3)$-invariants $i_k$ given
in Appendix A. We shall write $u=u_1+iu_2$, $u_1,u_2\in\bR$,
and similarly for other variables.

For the first subspace we define
\begin{eqnarray*}
P &=& (2i_4+3i_5-i_6)^2 +4i_1(i_2-i_3)(i_1i_3-6i_5) +4i_1^2(i_1i_5+i_8-i_7) \\
&& +4i_1i_2(i_6-i_4)+4i_8(5i_3-4i_2)+16i_3^2(2i_2-i_3) \\
&& +4i_7(2i_2-3i_3) +4i_2^2(i_2-5i_3)+8(i_1i_{11}-i_{13}).
\end{eqnarray*}
A computation using {\sc Maple} shows that for
$$ A=\left[ \begin{array}{ccc} 0&0&x\\ 0&z&y\\ u&v&-z
\end{array} \right] \in V $$
we have
$$ P(A)=(|u|^2-|x|^2)^2\left(
|x-\bar{u}|^2z_1^2-4(u_1x_2+u_2x_1)z_1z_2
+|x+\bar{u}|^2 z_2^2 \right)^2 . $$
On the other hand, we have $P(D)=-64(u_1v_2-u_2v_1)^2$.

For the second subspace we define $P=i_1^2+4(i_3-i_2)$.
One can easily verify that for
$$ A=\left[ \begin{array}{ccc} 0&x&y\\ u&z&0\\ v&0&-z
\end{array} \right] \in V $$
we have
$$ P(A)=(|u|^2+|v|^2-|x|^2-|y|^2)^2, $$
while $P(D)=-4(u_1v_2-u_2v_1)^2$.

For the remaining subspaces we have more elementary arguments.

In the third case let
$D=\diag(1,\zeta,\zeta^2)\in L(3)$, $\zeta=(-1+i\sqrt{3})/2$.
Since $D$ is a normal matrix, its field of values, $F(D)$,
is the equilateral triangle with vertices $1,\zeta,\zeta^2$.
Assume that $D$ is unitarily similar to a matrix $A=[a_{ij}]\in L_I$.
As $A\in L_I$,  $a_{22}$ is an eigenvalue of $A$ and so
$a_{22}\in\{1,\zeta,\zeta^2\}$. As $\tr(A)=0$ and $a_{11}=0$, we
deduce that $a_{33}=-a_{22}\notin F(D)$. Since $F(D)=F(A)$ and
$a_{33}$ is a diagonal entry of $A$, we have a contradiction.

In the last case let $D=\diag(1,i,-1-i)\in L(3)$.
Assume that $D$ is unitarily similar to a matrix
$$ A=\left[ \begin{array}{ccc} 0&0&x\\ y&z&0\\ u&v&-z
\end{array} \right] \in V. $$
Since $D$ is a normal matrix, so is $A$. From $AA^*=A^*A$ we obtain
the system of equations
\begin{eqnarray*}
&& x\bar{z}=z\bar{u},\quad z\bar{y}=-v\bar{u},\quad y\bar{u}=-2z\bar{v},  \\
&& |x|^2=|y|^2+|u|^2,\quad |y|=|v|.
\end{eqnarray*}
Assume that $z=0$. Then $uy=0$. As $A$ must be nonsingular, we have
$y\ne0$ and $u=0$. Thus $A^3$ is a scalar matrix. Since $D^3$ is not,
we have a contradiction.

We conclude that $z\ne0$. The equation $x\bar{z}=z\bar{u}$ implies
that $|x|=|u|$, which entails that $y=v=0$. Hence $z$ is an eigenvalue
of $A$, and so $z\in\{1,i,-1-i\}$. By switching the first two rows
(and columns) of $A$, we obtain the direct sum $[z]\oplus B$
where $B=\left[ \begin{array}{cc} 0&x\\ u&-z \end{array} \right].$
Hence $0\in F(B)$. This is a contradiction since $B$ is normal, and so
$F(B)$ is the line segment joining two of the eigenvalues of $D$.
\end{proof}

The following theorem provides an infinite collection of universal
patterns. It is an easy consequence of a result of Ko\v{s}ir and
Sethuraman proven in \cite{HS}.

\begin{theorem} \label{Kos-Set}
The pattern
$$ \left( (0,1)+\NE_{n-1} \right) \cup
\left\{ (i,1): 2<i\le n \right\} \cup \left\{ (n,2) \right\},
\quad n\ge3, $$
is $n$-universal.
\end{theorem}
\begin{proof}
Denote this pattern by $J$.
Let $A\in M(n)$ be arbitrary. By \cite[Theorem A.4]{HS}, there exists
$S\in\GL(n,\bC)$ such that $SA^*S^{-1}\in M_{J'}(n)$ and
$SAS^{-1}\in M_{J''}(n)$, where
$J'=(1,0)+\SW_{n-1}$ and $J''=\{(i,1):2<i\le n\}\cup\{(n,2)\}$.
By \cite[Remark 1]{HS}, we can assume
that $S$ is unitary.
As  $J=(J')^T \cup J''$, we deduce that $SAS^{-1}\in M_J(n).$
Hence $J$ is $n$-universal.
\end{proof}

\section{Nonsingular patterns} \label{Nonsing-Obl}

Before defining the singular and nonsingular patterns we introduce
some preliminary notions.

Let $\bR[x_1,x_2,\ldots]$ resp. $\bR[x_1^{\pm1},x_2^{\pm1},\ldots]$
be the polynomial ring resp. Laurent polynomial ring in countably
many commuting independent variables $x_1,x_2,\ldots$ over $\bR$.
We introduce an inner product, $\la .,. \ra$, in
$\bR[x_1^{\pm1},x_2^{\pm1},\ldots]$
by declaring that the basis consisting of the Laurent monomials
is orthonormal. We also introduce the involution $f\to f^*$ and the
shift endomorphism $\tau$ of this Laurent polynomial ring where,
by definition,
$$ f^*(x_1,x_2,\ldots) = f(x_1^{-1},x_2^{-1},\ldots) $$
and $\tau(x_i)=x_{i+1}$ for all $i\ge 1$.
For any Laurent polynomial $f$, we denote by $\CT\{f\}$ the constant
term of $f$. It is easy to see that for
$f,g\in\bR[x_1^{\pm1},x_2^{\pm1},\ldots]$ we have
$$ \la f,g \ra = \CT\{f^*g\} = \CT\{fg^*\}. $$

We shall denote by $\partial_i$ the partial derivative with respect
to the variable $x_i$. For any $f\in\bR[x_1,x_2,\ldots]$ we set
$$ \pd_f=f(\pd_1,\pd_2,\ldots). $$
To each pattern $I$ we associate the polynomial
$$ \chi_I=\prod_{(i,j)\in I}(x_i-x_j), $$
and  the differential operator
$$ \pd_ I=\pd_{\chi_I}=\prod_{(i,j)\in I}(\pd _i-\pd _j). $$

Next we introduce special notation for certain symmetric polynomials in
the first $n$ variables $x_1,x_2,\ldots,x_n$:
\begin{equation} \label{ElemSim}
\sig_{k,n}=\sum_{1\le i_1<i_2<\cdots<i_k\le n}
x_{i_1}x_{i_2}\cdots x_{i_k}, \quad k\in\{0,1,\ldots,n\},
\end{equation}
for the elementary symmetric functions, and
\begin{equation} \label{PotpSim}
h_{k,n}=\sum_{d_1+\cdots+d_n=k} x_1^{d_1}x_2^{d_2}\cdots x_n^{d_n},
\quad k\in\{0,1,\ldots,n\},
\end{equation}
for the complete symmetric functions.

We denote by $\pH(n)$ the quotient of
$\bR[x_1,x_2,\ldots,x_n]$ modulo the ideal $\pK(n)$ generated by the
$\sig_{k,n}$, $k\in\bZ_n$, and define
$\vf_n:\bR[x_1,x_2,\ldots,x_n]\to\pH(n)$ to be the
natural homomorphism.

\begin{definition}
We say that $I\in\pP_n$ is {\em $n$-singular} if
$\vf_n(\chi_I)=0$, and otherwise that it is {\em $n$-nonsingular}.
We say that $I\in\pP'_n$ is $n$-{\em exceptional} if it is
$n$-singular but not $n$-defective.
\end{definition}

Note that if a pattern $I$ contains a diagonal position
then $\chi_I=0$ (in particular, $I$ is singular).

\begin{lemma} \label{Pros-Lema}
Every $n$-nonsingular pattern is also $m$-nonsingular for all $m>n$.
\end{lemma}
\begin{proof}
Let $I\in\pP_n$ be $m$-singular for some $m>n$.
Then there exist polynomials
$f_k\in\bR[x_1,x_2,\ldots,x_m]$ such that
$$ \chi_I(x_1,x_2,\ldots,x_n)=\sum_{k=1}^m
f_k(x_1,x_2,\ldots,x_m)\sig_{k,m}(x_1,x_2,\ldots,x_m) . $$
By setting $x_{n+1}=\cdots=x_m=0$, we see that $I$ is $n$-singular.
\end{proof}

We remark that the converse of this lemma is not valid. (All seven
patterns listed in Table 2 of Section \ref{Cetiri} are
counter-examples for $n=4$.)

To simplify the notation, we set
$$ \chi_n=\chi_{\NE_n}=\prod_{1\le i<j\le n} (x_i-x_j). $$
Its expansion is given by the well known formula
\begin{equation} \label{Hi_n}
\chi_n = \sum_{\sig\in S_n} \sgn(\sig) \prod_{i=1}^n x^{n-i}_{\sig(i)}.
\end{equation}
Consequently, we have $\la \chi_n,\chi_n \ra =n!$.

The following result (see \cite[Proposition 2.3]{AD}) provides a
practical method for deciding whether a pattern is nonsingular.
In particular, it implies that $\NE_n$ is nonsingular.

\begin{proposition} \label{FundPol}
A pattern $I\in\pP'_n$ is $n$-singular iff $\la \chi_I,\chi_n \ra =0$.
\end{proposition}

We now state two problems concerning the inner product in this
proposition.

\begin{problem} \label{Nejed}
Is it true that $\la \chi_I,\chi_n \ra \le n!$ for all $I\in\pP_n$
and that equality holds iff $\chi_I=\chi_n$?
\end{problem}
\begin{problem} \label{Norma}
Is it true that $\la \chi_I,\chi_I \ra \ge n!$ for all $I\in\pP'_n$
and that equality holds iff $\chi_I=\pm\chi_n$?
\end{problem}
The computations carried out for $n\le5$ show that, in these cases,
the answer is affirmative for both problems.

In connection with Theorem \ref{Kos-Set} we propose
\begin{conjecture} \label{Hess}
Let $J_{k,n}\in\pP'_n$, $k\in\bZ_{n-1}$, be the union of the translate
$(0,1)+\NE_{n-1}$, the product $\{n\}\times\bZ_k$, and
$\{(i,1):k<i<n\}$. Then
$$ \la \chi_{J_{k,n}}, \chi_n \ra = (-1)^{n-1}\binom{n}{k}\binom{n-2}{k-1}. $$
\end{conjecture}
The assertion in the case $k=1$ will be proved in Example \ref{Primer-1}.
The universality of $J_{2,n}$ was proved in Theorem \ref{Kos-Set}.
The conjecture has been verified for $n\le10$ and all $k\in\bZ_{n-1}$,
except for $(n,k)=(10,5)$ in which case our program ran out of memory.

We say that a pattern $I$ is {\em simply laced}, or just {\em simple},
if $I \cap I^T=\emptyset$,
and otherwise we say that $I$ is {\em doubly laced}.
In particular, a simple pattern is strict.
Any simple pattern $I\in\pP_n$ can be extended to
a simple pattern $I'\in\pP'_n$. As $\chi_I$ divides $\chi_{I'}$
and $\chi_{I'}=\pm\chi_n$, any simple pattern is nonsingular.

The following two results are extracted from Proposition 4.3 and
Theorem 4.2 of our paper \cite{AD}.

\begin{proposition} \label{Prosiri}
Every nonsingular pattern $I\in\pP_n$ can be extended to
a nonsingular pattern $I'\in\pP'_n$.
\end{proposition}

\begin{theorem} \label{TeorAD} {\rm (Generalization of Schur's
triangularization theorem)}
Every nonsingular pattern in $\pP_n$ is $n$-universal.
In particular, every simple pattern in $\pP_n$ is $n$-universal
\end{theorem}

In connection with these results we pose an open problem.

\begin{problem} \label{Pros-Str}
Let $I$ be a strict $n$-universal pattern.
Does $I$ extend to a strict $n$-universal pattern of size $\mu_n$?
\end{problem}

If we replace the word ``strict'' with ``proper'' then
Proposition \ref{3-Obl} shows that the answer is negative.

A special case of the general inner product identity that we prove in the next
theorem will be used in the proof of the subsequent proposition.

\begin{theorem} \label{KompGr}
Let $G$ be a connected compact Lie group, $T$ a maximal torus with Lie algebra
$\mathfrak{t}$ and $W$ the Weyl group.
Denote by $\mu$ be the number of positive roots and by $\chi$ their product.
Let $\pA=\oplus_{n\ge0} \pA_n$ be the algebra of polynomial functions $\mathfrak{t}\to\bR$ with the degree gradation,
and let $\la\cdot,\cdot\ra$ be a $W$-invariant inner product on $\pA$.
If $h\in\mathcal{A}$ is $W$-invariant, then
$$ \|\chi\|^2 \la hf,h\chi \ra=\|h\chi\|^2\la f,\chi \ra,
\quad \forall f\in\mathcal{A}_\mu. $$
\end{theorem}
\begin{proof}
We may assume that $h\ne0$.
For the linear functional $L:\pA_\mu\to\bR$ defined by $L(f)=\la hf, h\chi \ra$,
we have $\sig\cdot L=\sgn(\sig)L$ for all $\sig\in W$. There is a unique $g\in\pA_\mu$
such that $L(f)=\la f,g \ra$ for all $f\in\pA_\mu$. Moreover, $\sig\cdot g=\sgn(\sig)g$
for all $\sig\in W$. As the sign representation of $W$ occurs only once in $\pA_\mu$,
we must have $g=c\chi$ for some $c\in\bR$. The equality $L(\chi)=\la \chi,c\chi \ra$
completes the proof.
\end{proof}

The next proposition, a nonsingular analogue of Proposition \ref{UnivPros},
shows that some extensions of nonsingular
patterns are nonsingular. In a weaker form, it was originally
conjectured by Jiu-Kang Yu.

\begin{proposition} \label{Pros-Stav}
For $I\in\pP'_n$, $J\in\pP'_{m-n}$, $m>n$, and
$$ I'=I \cup \left( (n,n)+J \right) \cup
\left( (0,n)+\bZ_n\times\bZ_{m-n} \right) $$
we have
$$ \la \chi_{I'},\chi_m \ra={m \choose n} \la \chi_I,\chi_n \ra
\la \chi_J,\chi_{m-n} \ra. $$
\end{proposition}

\begin{proof}
We apply the above theorem to the case where $G=\Un(n)\times\Un(m-n)$.
Then the algebra $\pA$ can be identified with the polynomial algebra
$\bR[x_1,\ldots,x_m]$ so that $\bR[x_1,\ldots,x_n]$ resp.
$\bR[x_{n+1},\ldots,x_m]$ is the corresponding algebra for
$\Un(n)$ resp. $\Un(m-n)$. The polynomial $\chi$ factorizes as
$\chi=\chi_n\cdot\tau^n\chi_{m-n}$, where $\tau$ is the shift operator.
We take $f=\chi_I\cdot\tau^n\chi_J$ and for $h$ we take the polynomial
$$
\prod_{i=1}^n\prod_{j=n+1}^m(x_i-x_j),
$$
which is obviously invariant under $W=S_n\times S_{m-n}$.
Since $h\chi=\chi_m$ and $hf=\chi_{I'}$, the theorem gives the identity
$$ \|\chi\|^2 \la \chi_{I'},\chi_m \ra=\|\chi_m\|^2\la f,\chi \ra. $$
Since $\|\chi\|^2=n!(m-n)!$, $\|h\chi\|^2=\|\chi_m\|^2=m!$, and
$$ \la \chi_I\cdot\tau^n\chi_J,\chi_n\cdot\tau^n\chi_{m-n} \ra=
 \la \chi_I,\chi_n \ra \la \chi_J,\chi_{m-n} \ra, $$
the assertion follows.
\end{proof}

\section{Equivalence and weak equivalence} \label{Ekv}

The symmetric group $S_n$ acts on $\pP_n$ by
$\sig(I)=\{(\sig(i),\sig(j)):\, (i,j)\in I\}$ for $\sig\in S_n$.
For $\sig\in S_n$ and $I\in\pP_n$ we have $\sig\cdot\chi_I=\chi_{\sig(I)}$.
As $\pP'_n$ is $S_n$-invariant, we obtain an action on $\pP'_n$.
We denote by $\tilde{S}_n$ the group of transformations of
$\pP'_n$ generated by the action of $S_n$
and the restriction of transposition $T$ to $\pP'_n$.
(Note that this restriction commutes with the action of $S_n$.)
As the inner product is $S_n$-invariant, we have
$$ \la \chi_{\sig(I)},\chi_n \ra=\la \chi_I,\sig\cdot\chi_n \ra=
\sgn(\sig) \la \chi_I,\chi_n \ra,\quad I\in\pP'_n,\quad \sig\in S_n. $$

We say that the patterns $I,I'\in\pP'_n$ are \emph{equivalent}
if they belong to the same orbit of $\tilde{S}_n$.
If so, we shall write $I\approx I'$. We denote by $[I]$
the equivalence class of $I\in\pP'_n$.

Assume that $I\approx I'$. It is easy to see that
$I'$ is universal iff $I$ is universal. Since $\ker \vf_n$ is
invariant under permutations of the variables $x_1,x_2,\ldots,x_n$,
we deduce that $I'$ is nonsingular iff $I$ is nonsingular.
We say that the class $[I]$ is {\em singular, nonsingular,
universal, defective} or {\em exceptional} if $I$ has
the same property. These terms are clearly well defined.

Let $I\in\pP$ be any pattern and let us fix a position $(i,j)\in I$ such
that $(j,i)\notin I$. Denote by $I'$ the pattern obtained from $I$ by
replacing the position $(i,j)$ with $(j,i)$. We shall refer to
the transformation $I\to I'$ as a {\em flip}.

\begin{problem} \label{Flip}
Let $I\in\pP'_n$ be $n$-universal and let $I\to I'$ be a flip.
Is it true that $I'$ is $n$-universal?
\end{problem}

We say that the patterns $I,I'\in\pP_n$ are \emph{weakly equivalent}
if $I$ can be transformed to $I'$ by using flips and the
action of $S_n$. If so, we shall write $I\sim I'$. We denote by $[I]_w$
the weak equivalence class of $I\in\pP_n$.
As the transposition map $\pP_n\to\pP_n$ can be realized
by a sequence of flips, we have $[I]\subseteq[I]_w$ for all $I\in\pP_n$.
Note that the nonsingularity property is preserved by weak equivalence.

Let us define the {\em complexity}, $\nu(I)$, of a pattern $I$
as the number of positions $(i,j)\in I\cap I^T$ with $i\le j$.
The patterns of complexity 0 are precisely the simple patterns.
Observe that a pattern $I\in\pP'_n$ has complexity 1 iff there exist
a unique $2$-element subset $\{i,j\}\subseteq\bZ_n$ such that
$(i,j)$ and $(j,i)$ belong to $I$.
Similarly, $I\in\pP'_n$ has complexity 1 iff there exist a unique
$2$-element subset $\{k,l\}\subseteq\bZ_n$ such that neither
$(k,l)$ nor $(l,k)$ is in $I$.

It is natural to ask which patterns $I\in\pP'_n$ of
complexity 1 are universal or nonsingular. We shall now give a
complete answer to the latter question. At the same time we classify,
up to weak equivalence, the patterns in $\pP'_n$ having complexity 1.

\begin{theorem} \label{WE-Compl-1}
Let $I\in\pP'_n$, $\nu(I)=1$, and let $\{i,j\}$ resp.
$\{k,l\}$ be the unique $2$-element subset of $\bZ_n$ such that
$(i,j),(j,i)\in I$ resp. $(k,l),(l,k)\notin I$.

(a) If $i,j,k,l$ are not distinct then
$$ I \sim \left( \NE_n \setminus \{(1,2)\} \right)
\cup \{(3,1)\}, \quad n\ge3, $$
and we have $\la \chi_I,\chi_n \ra= \pm\; n!/2.$

(b) If  $i,j,k,l$ are distinct then $I$ is singular and
$$ I \sim \left( \NE_n \setminus \{(1,2)\} \right)
\cup \{(4,3)\}, \quad n\ge4. $$
\end{theorem}
\begin{proof}
(a) We have, say, $i=k$. Let $J=\sig(I)$, where $\sig\in S_n$
is chosen so that $\sig(i)=1,$ $\sig(j)=3$ and $\sig(l)=2$. For each
$(r,s)\in J$ with $r>s$ and $(r,s)\ne(3,1)$ we apply a flip to
replace $(r,s)$ with $(s,r)$. We obtain the desired pattern.
Proposition \ref{Pros-Stav} gives the formula for the inner product.

(b) The equivalence assertion is proved in the same manner as in (a)
and, since this new pattern is defective, $I$ must be singular.
\end{proof}

We conclude this section by providing some numerical data
about the equivalence classes in $\pP'_n$ for $2\le n\le 5$.
For each $n$, Table 1 gives the cardinality of $\pP'_n$, the
number of equivalence classes, and the number of nonsingular,
defective, and exceptional classes in that order. The
last column records the number of weak equivalence classes.

\begin{center}
\begin{tabular}{crrrrrr}
&\\
\multicolumn{6}{c}
{\bf  Table 1: Equivalence classes} \\
&\\
$n$ & $|\pP'_n|$ & Equ. & Nons. & Def. & Exc. & Weak \\
&\\ \hline &\\
2 & 2 & 1 & 1 & 0 & 0 & 1 \\
3 & 20 & 3 & 3 & 0 & 0 & 2 \\
4 & 928 & 30 & 19 & 4 & 7 & 12 \\
5 & 184956 & 880 & 619 & 66 & 195 & 110 \\
\end{tabular}
\end{center}
For any $n$ we have $|\pP'_n|=\binom{2\mu_n}{\mu_n}$, but
a formula for the number of (weak) equivalence classes is not known.

\begin{problem} \label{Equ-Cl}
Find a formula for the number of (weak) equivalence classes in $\pP'_n$.
\end{problem}

\section{An analogue of Schur's theorem} \label{Analog}

Recall the patterns $\Lambda_n\in\pP'_n$, $n\ge1$, defined in the
Introduction. The zero entries required by $\Lambda_n$ for
$n=2,3,\ldots,7$ are exhibited below:
\begin{equation*}
\left[ \begin{array}{cc}
* & 0 \\ * & *  \end{array} \right],\quad
\left[ \begin{array}{ccc}
* & 0 & 0 \\ * & * & * \\ 0 & * & * \end{array} \right],\quad
\left[ \begin{array}{cccc}
* & 0 & * & 0 \\ 0 & * & 0 & * \\ * & 0 & * & * \\ 0 & * & * & *
\end{array} \right],\quad
\left[ \begin{array}{ccccc}
* & 0 & 0 & * & 0 \\ 0 & * & 0 & 0 & * \\ 0 & 0 & * & * & * \\
* & 0 & * & * & * \\ 0 & * & * & * & *
\end{array} \right],
\end{equation*}
\begin{equation*}
\left[ \begin{array}{cccccc}
* & 0 & 0 & 0 & * & 0 \\ 0 & * & 0 & 0 & 0 & * \\
0 & 0 & * & 0 & * & * \\ 0 & 0 & * & * & * & * \\
* & 0 & * & * & * & * \\ 0 & * & * & * & * & *
\end{array} \right],\quad
\left[ \begin{array}{ccccccc}
* & 0 & 0 & 0 & 0 & * & 0 \\ 0 & * & 0 & 0 & 0 & 0 & * \\
0 & 0 & * & 0 & 0 & * & * \\ 0 & 0 & * & * & * & * & * \\
0 & 0 & 0 & * & * & * & * \\ * & 0 & * & * & * & * & * \\
0 & * & * & * & * & * & *
\end{array} \right].
\end{equation*}
The maximum of $\nu(I)$ over all $I\in\pP'_n$ is
$[\mu_n/2]$. Let us write $n=4m+r$ where $m$ is a nonnegative integer
and $r\in\{0,1,2,3\}.$ Note that $\Lambda_n$ is symmetric for
$r\in\{0,1\}$, and otherwise there is a unique $(i,j)\in\Lambda_n$,
namely $(2m+1,2m+2)$, such that $(j,i)\notin\Lambda_n$. Hence
$\Lambda_n$ has the maximal complexity $[\mu_n/2]$.
Our objective in this section is to prove that $\Lambda_n$ is
nonsingular. As $\Lambda_n$ is a (necessary) minor modification of the
triangular pattern $\NW_n$, we consider this result as an analogue
of Schur's theorem.

\begin{theorem} \label{SchurAnalog}
We have $ \la \chi_{\Lambda_n}, \chi_n \ra = (-1)^s n!/2^s,$ where
$s=[(n+1)/4]$. In particular, $\Lambda_n$ is universal.
\end{theorem}

For the proof we need four lemmas and the following three facts which
follow immediately from \cite[Chapter III, Lemma 3.9]{He}:

(i) $x_i^k\in\pK(n)$ for $k\ge n$, $1\le i\le n$;

(ii) If $f,g\in\bR[x_1,x_2,\ldots,x_n]$ and $f\equiv g \pmod{\pK(n)}$
then
\begin{equation} \label{fg-hi}
\pd_f\chi_n=\pd_g\chi_n;
\end{equation}

(iii) If $f\in\bR[x_1,x_2,\ldots,x_n]_{\mu_n}$  then
\begin{equation} \label{df-hi}
\pd_f\chi_n=\prod_{k=1}^{n-1}k!\cdot\la f,\chi_n \ra.
\end{equation}

By diferentiating the formula (\ref{PotpSim}), we obtain that
\begin{equation} \label{IzvPot}
\pd_i h_{k,n}=\sum_{d_1+\cdots+d_n=k-1} (1+d_i)x_1^{d_1}x_2^{d_2}\cdots
x_n^{d_n}, \quad i\le n.
\end{equation}

We prove first the following congruence.

\begin{lemma}\label{Kongr}
For $1\le r\le m\le n$ we have
$$
\prod_{m<i\le n} (x_r-x_i) \equiv \pd_r h_{n-m+1,m} \pmod{\pK(n)}.
$$
\end{lemma}
\begin{proof}
Without any loss of generality we may assume that $r=1$.
Let $s,t$ be two additional commuting indeterminates. Observe that
for any polynomial $f(t)$, with coefficients in $\bR[x_1,\ldots,x_n]$,
the constant term of
$$ F(s,t)=f(s^{-1})\cdot\frac{1}{1-st}, $$
when expanded into a formal Laurent series with respect to $s$,
is equal to $f(t)$. For
$$ f(t)=(t-x_1)\cdot\prod_{m<i\le n} (t-x_i) $$
we have
\begin{eqnarray*}
F(s,t) &=& \frac{\prod_{i=1}^n (1-sx_i)}
{s^{n-m+1} (1-st) \prod_{i=2}^m (1-sx_i)} \\
&=& \frac{1}{s^{n-m+1}}\cdot\sum_{k=0}^n (-s)^k \sig_{k,n}\cdot
\sum_{l=0}^\infty h_{l,m}(t,x_2,\ldots,x_m)s^l.
\end{eqnarray*}
Hence
$$ f(t)=\sum_{k=0}^{n-m+1} (-1)^k \sig_{k,n} h_{n-m+1-k,m}
(t,x_2,\ldots,x_m). $$
By evaluating the partial derivative with respect to $t$ at the
point $t=x_1$, we obtain that
$$ \prod_{m<i\le n} (x_1-x_i) = \sum_{k=0}^{n-m+1} (-1)^k \sig_{k,n}
\partial_1 h_{n-m+1-k,m}(x_1,x_2,\ldots,x_m) $$
and the assertion of the lemma follows.
\end{proof}

Recall the endomorphism $\tau$ defined in the beginning of
Section \ref{Nonsing-Obl}.

\begin{lemma} \label{L:1}
For $P=\pd_1 h_{2,n-1}(x_1,x_n)$, i.e.,
$$P=\sum_{k=0}^{n-2}(n-1-k)x_1^{n-2-k}x_n^k,\quad n\ge4,$$
we have
\begin{equation}\label{E:1}
\partial_P^2\chi_n=(-1)^nn!(n-2)!\left(\frac{n-1}{2}x_1+
\frac{n-3}{2}x_n-\sum_{k=2}^{n-1}x_k\right) \tau\chi_{n-2}.
\end{equation}
\end{lemma}

\begin{proof}
Since $x_1^n,x_n^n\in\pK(n)$, we have
$$P^2 \equiv ax_1^{n-3}x_n^{n-1}+bx_1^{n-2}x_n^{n-2}+cx_1^{n-1}x_n^{n-3},
\pmod{\pK(n)},$$
where
\begin{align*}
a &=\sum_{k=1}^{n-2}k(n-1-k)=n(n-1)(n-2)/6, \\
b &=\sum_{k=1}^{n-1}k(n-k)=n(n-1)(n+1)/6, \\
c &=\sum_{k=2}^{n-1}k(n+1-k)=n(n+5)(n-2)/6.
\end{align*}
By using the property (\ref{fg-hi}), we obtain that
$$\pd_P^2\chi_n=(a\pd_1^{n-3}\pd_n^{n-1}+b\pd_1^{n-2}\pd_n^{n-2}
+c\pd_1^{n-1}\pd_n^{n-3})\chi_n.$$
If we omit from $\chi_n$ the terms $\pm x_1^{d_1}\cdots x_n^{d_n}$ with
$d_1+d_n<2n-4$, we obtain the polynomial $(-1)^n Q \tau\chi_{n-2}$ where
$$ Q=(x_1^{n-1}x_n^{n-2}-x_1^{n-2}x_n^{n-1})
-(x_1^{n-1}x_n^{n-3}-x_1^{n-3}x_n^{n-1})\sum_{k=2}^{n-1}x_k.
$$
Since the omitted terms are killed by $\pd_P^2$, we have
$\pd_P^2\chi_n=(-1)^n C \tau\chi_{n-2}$ where
$$ C=(n-1)!\left[ (n-2)!((b-a)x_1+(c-b)x_n)
-(n-3)!(c-a)\sum_{k=2}^{n-1}x_k \right]. $$
It remains to plug in the values of $a,b$ and $c$.
\end{proof}

The next lemma gives another important identity.

\begin{lemma}\label{L:2}
Let $J_n=J'_n\cup (J'_n)^T$, $n\ge4$, where $J'_n$ is the union of
$\{(1,2),(1,n-1),(2,n)\}$ and $\{1,2\}\times\{3,4,\ldots,n-2\}$.
Then
$$\pd_{J_n}\chi_n=\frac{1}{2}n!(n-1)!(n-2)!(n-3)!\tau^2\chi_{n-4}.$$
\end{lemma}

\begin{proof}
By using Lemma \ref{Kongr} we obtain the congruence
\begin{align*}
\chi_{J_n} =& -\left(\prod_{i=2}^{n-1}(x_1-x_i)\right)^2\left((x_2-x_n)
\prod_{i=3}^{n-2}(x_2-x_i)\right)^2 \\
\equiv & -P^2R^2 \pmod{\pK(n)},
\end{align*}
where $P$ is defined as in the previous lemma and
$$R=\pd_2 h_{n-2,3}(x_1,x_2,x_{n-1})=
\sum_{d_1+d_2+d_3=n-3}(1+d_2)x_1^{d_1}x_2^{d_2}x_{n-1}^{d_3}.$$
In view of the formula (\ref{E:1}), it suffices to prove that
\begin{equation}\label{E:2}
\pd_R^2F=\frac{(-1)^{n-1}}{2}(n-1)!(n-3)!\tau^2\chi_{n-4},
\end{equation}
where
$$F=\left(\frac{n-1}{2}x_1+\frac{n-3}{2}x_n-
\sum_{k=2}^{n-1}x_k\right)\tau\chi_{n-2}.$$

Since $\deg(R^2)=2n-6$, we need only consider the terms
in $F$ for which the sum of the exponents of $x_1, x_2$ and
$x_{n-1}$ is at least $2n-6$. Their sum is
$$F'=(-1)^n(\frac{n-1}{2}x_1-x_2-x_{n-1})
(x_2^{n-3}x_{n-1}^{n-4}-x_2^{n-4}x_{n-1}^{n-3})
\tau^2\chi_{n-4}.$$
Note also that the exponents of $x_1$ in $F'$ are $\le1$. So we need
only consider the terms in $\pd_R^2$ for which the exponent
of $\pd_1$ is 0 or 1. Their sum is $D_1^2+2\pd_1D_1D_2$ where
$$D_1=\sum_{k=0}^{n-3}(n-2-k)\pd_2^{n-3-k}\pd_{n-1}^k,$$
$$D_2=\sum_{k=0}^{n-4}(n-3-k)\pd_2^{n-4-k}\pd_{n-1}^k.$$
After expanding the products $D_1D_2$ and $D_1^2$:
\begin{eqnarray*}
D_1D_2 &=& a_1\pd_2^{n-3}\pd_{n-1}^{n-4}+b_1\pd_2^{n-4}\pd_{n-1}^{n-3}+\ldots, \\
D_1^2 &=& a_2\pd_2^{n-2}\pd_{n-1}^{n-4}+b_2\pd_2^{n-4}\pd_{n-1}^{n-2}+\ldots,
\end{eqnarray*}
the exhibited coefficients can be easily computed:
\begin{eqnarray*}
a_1 &=& \sum_{k=1}^{n-3}k(n-1-k)=(n-2)(n-3)(n+2)/6, \\
a_2 &=& \sum_{k=2}^{n-2}k(n-k)=(n-1)(n-3)(n+4)/6, \\
b_1=b_2 &=& \sum_{k=1}^{n-3}k(n-2-k)=(n-1)(n-2)(n-3)/6.
\end{eqnarray*}
Thus
\begin{eqnarray*}
D_1D_2(x_2^{n-3}x_{n-1}^{n-4}-x_2^{n-4}x_{n-1}^{n-3}) &=&
(n-3)!(n-4)!(a_1-b_1) \\
&=& (n-2)!(n-3)!/2, \\
D_1^2(x_2^{n-2}x_{n-1}^{n-4}-x_2^{n-4}x_{n-1}^{n-2}) &=&
(n-2)!(n-4)!(a_2-b_2) \\
&=& (n-1)!(n-3)!.
\end{eqnarray*}
Hence
\begin{eqnarray*}
\pd_R^2F &=& (D_1^2+2\pd_1D_1D_2)F'\\
&=& (-1)^n \left[ (n-1)D_1D_2(x_2^{n-3}x_{n-1}^{n-4}-x_2^{n-4}x_{n-1}^{n-3}) \right. \\
& & \quad \left. -D_1^2(x_2^{n-2}x_{n-1}^{n-4}-x_2^{n-4}x_{n-1}^{n-2}) \right]
\tau^2\chi_{n-4},
\end{eqnarray*}
which completes the proof.
\end{proof}

The fourth lemma follows easily from the previous one.
\begin{lemma}\label{L:3}
Let $I=J_n \cup \left( (2,2)+I' \right)$, $n\ge4$, where $J_n$ is defined
as in the previous lemma and $I'\in\pP'_{n-4}$ is arbitrary. Then
$$ \la \chi_I, \chi_n \ra = \frac{1}{2}n(n-1)(n-2)(n-3)\la \chi_{I'},
\chi_{n-4} \ra. $$
\end{lemma}

\begin{proof}
Since $\chi_{I}=\chi_{J_n} \tau^2 \chi_{I'}$, we have
$\pd_I=\pd_{J_n}\pd_{\tau^2 \chi_{I'}}$.
By using (\ref{df-hi}) and Lemma \ref{L:2}, we obtain that
\begin{eqnarray*}
\la\chi_I,\chi_n\ra &=& \frac{\pd_I\chi_n}{\prod_{k=1}^{n-1}k!} \\
&=& \frac {n(n-1)(n-2)(n-3) \tau^2 \left( \pd_{I'} \chi_{n-4} \right) }
 {2\prod_{k=1}^{n-5}k!} \\
&=& \frac{1}{2}n(n-1)(n-2)(n-3)\la \chi_{I'}, \chi_{n-4} \ra.
\end{eqnarray*}
\end{proof}

We can now prove the theorem itself.
\begin{proof}
We construct the patterns $\Lambda'_n$ inductively as
$\Lambda'_0=\Lambda'_1=\emptyset$, $\Lambda'_2=\{(1,2)\}$,
$\Lambda'_3=\{(2,1),(2,3),(3,2)\}$, and
$\Lambda'_n=J_n\cup((2,2)+\Lambda'_{n-4})$ for $n\geq4$. We claim
that $\la \chi_{\Lambda'_n}, \chi_n \ra =n!/2^s$.

We prove the claim by induction on $n$. It is straightforward to
verify the claim for $n=0,1,2,3$. For $n\ge 4$, by
Lemma \ref{L:3} and the induction hypothesis, we have
\begin{eqnarray*}
\la \chi_{\Lambda'_n}, \chi_n \ra &=& \frac{1}{2}n(n-1)(n-2)(n-3)
\la \chi_{\Lambda'_{n-4}}, \chi_{n-4} \ra \\
&=& \frac{1}{2}n(n-1)(n-2)(n-3)\cdot\frac{(n-4)!}{2^{s-1}}=\frac{n!}{2^s}.
\end{eqnarray*}
It remains to observe that $\sig(\Lambda'_n)=\Lambda_n$, where
$$ \sig=\prod_{k=1}^s(2k-1,2k) \in S_n. $$
\end{proof}

\section{A remarkable family of nonsingular patterns} \label{Analogue}

It is a challenging problem to construct an infinite family
of new nonsingular doubly laced patterns.
The main result of this section gives a construction of such a family.
It includes, as a special case, a new analogue of Schur's theorem
namely another modification of the triangular pattern $\NW_n$
(see Proposition \ref{Primer-2}).

We first introduce the notation that we need to state and prove our theorem.
Let $\sig\in S_n$ and let ${\bf r}=(r_1,r_2,\ldots,r_{n-1})$ be a sequence
with $r_k\in\sig(\bZ_k)$ for all $k\in\bZ_{n-1}$.
For $k\in\bZ_{n-1}$ we set
$$ I_k=\{(r_k,\sig(j)): k<j\le n \}. $$
As usual, $\de_{ij}$ will denote the Kronecker delta symbol.

The main technical tool is the following lemma.
\begin{lemma} \label{glavna}
Under the above hypotheses, for $0\le m<n$ we have
\begin{eqnarray*}
&& \pd_{I_1}\pd_{I_2}\cdots\pd_{I_m}\chi_n= \\
&& \quad \sgn(\sig)\left(\prod_{k=1}^m
(n-k+\de_{r_k,\sig(k)})! \right) \chi_{n-m}(x_{\sig(m+1)},\ldots,x_{\sig(n)}).
\end{eqnarray*}
\end{lemma}
\begin{proof}
We use induction on $m$. The assertion obviously holds for $m=0$.
Assume that $m>0$. By the induction hypothesis we have
\begin{eqnarray*}
&& \pd_{I_1}\pd_{I_2}\cdots\pd_{I_m}\chi_n= \\
&& \quad \sgn(\sig)\left(\prod_{k=1}^{m-1}(n-k+\de_{r_k,\sig(k)})!\right)
\cdot\pd_{I_m}\chi_{n-m+1}(x_{\sig(m)},\ldots,x_{\sig(n)}).
\end{eqnarray*}
By Lemma \ref{Kongr} we have
\begin{eqnarray*}
\chi_{I_m} &=& \prod_{m<i\le n} (x_{r_m}-x_{\sig(i)}) \\
&\equiv& \pd_{r_m} h_{n-m+1,m}(x_{\sig(1)},\ldots,x_{\sig(m)})
\pmod{\pK(n)}.
\end{eqnarray*}
By invoking the property (\ref{fg-hi}) and the identity (\ref{df-hi}),
we deduce that
\begin{eqnarray*}
&& \pd_{I_m}\chi_{n-m+1}(x_{\sig(m)},\ldots,x_{\sig(n)})= \\
&& \quad \sum_{d_1+\cdots+d_m=n-m} (1+d_{r_m}) \pd_{\sig(1)}^{d_1} \cdots
\pd_{\sig(m)}^{d_m} \chi_{n-m+1}(x_{\sig(m)},\ldots,x_{\sig(n)}).
\end{eqnarray*}
All terms in this sum vanish except the one for $d_1=\cdots=d_{m-1}=0$
and $d_m=n-m$. We infer that
\begin{eqnarray*}
&& \pd_{I_m}\chi_{n-m+1}(x_{\sig(m)},\ldots,x_{\sig(n)})= \\
&& \quad \left(1+(n-m)\de_{r_m,\sig(m)} \right)\pd_{\sig(m)}^{n-m}
\chi_{n-m+1}(x_{\sig(m)},\ldots,x_{\sig(n)}).
\end{eqnarray*}
Since
$$ \pd_{\sig(m)}^{n-m} \chi_{n-m+1}(x_{\sig(m)},\ldots,x_{\sig(n)})=
(n-m)!\chi_{n-m}(x_{\sig(m+1)},\ldots,x_{\sig(n)}), $$
we are done.
\end{proof}

Let us fix a permutation $\sig\in S_n$ and let
${\bf i}=(i_1,i_2,\ldots,i_{n-1})$ be a sequence
of distinct integers such that
$|i_k|\in\sig(\bZ_k)$ for all $k\in\bZ_{n-1}$. Next we set
$$ (i_k,\sig(j))^+=\begin{cases}
(i_k,\sig(j)) & \text{if } i_k>0; \\
(\sig(j),-i_k) & \text{otherwise}.
\end{cases} $$
With these data at hand, we construct the strict pattern
\begin{equation} \label{Obl-J}
J(\sig,{\bf i})=\{(i_k,\sig(j))^+:1\le k<j\le n\}.
\end{equation}

We claim that the conditions imposed on {\bf i} imply that the map
sending $(k,j)\to(i_k,\sig(j))^+$ for $k<j$ is injective, i.e., that
$|J(\sig,{\bf i})|=\mu_n$. Indeed, assume that two different pairs
$(k,j)$ and $(r,s)$, with $k<j$ and $r<s$, have the same image, i.e.,
$(i_k,\sig(j))^+=(i_r,\sig(s))^+$.
Clearly, we must have $k\ne r$ and $i_k i_r<0$. Say, $k<r$.
Then $|i_k|=\sig(s)\notin\sig(\bZ_r)$, which contradicts the condition
$|i_k|\in\sig(\bZ_k)$. This proves our claim, and so we have
$J(\sig,{\bf i})\in\pP'_n$.

Let $\pJ_n\subseteq\pP'_n$ be the set of all patterns $J(\sig,{\bf i})$.
For any $n$, let $\iota\in S_n$ be the identity permutation.
There are exactly $(n!)^2$ choices for the ordered pairs
$(\sig,{\bf i})$. However the corresponding patterns $J(\sig,{\bf i})$
are not all distinct. For instance, if $n=2$ we have
$J(\sig,(2))=J(\iota,(-1))$ with $\sig=(1,2)$.

The following is the main result of this section.

\begin{theorem} \label{Fam-J}
For $J=J(\sig,{\bf i})$ we have
$$\la \chi_J,\chi_n \ra =(-1)^d\sgn(\sig)\prod_{k:|i_k|=\sig(k)}(n-k+1), $$
where
$$ d=\sum_{k \;:\; i_k<0}(n-k). $$
\end{theorem}
\begin{proof}
We apply the lemma with ${\bf r}=(|i_1|,|i_2|,\ldots,|i_{n-1}|)$
and $m=n-1$. Thus we now have
$I_k=\{(|i_k|,\sig(j)): k<j\le n \}$ for all $k\in\bZ_{n-1}$.
Since $\chi_1=1$, the lemma gives
$$ \pd_{I_1}\pd_{I_2}\cdots\pd_{I_{n-1}}\chi_n=
\sgn(\sig)\prod_{k=1}^{n-1}(n-k+\de_{|i_k|,\sig(k)})!. $$
By (\ref{df-hi}) we have
$$
\pd_{I_1}\pd_{I_2}\cdots\pd_{I_{n-1}}\chi_n=\prod_{k=1}^{n-1}(n-k)!\cdot
\la \chi_{I_1}\chi_{I_2}\cdots\chi_{I_{n-1}}, \chi_n \ra.
$$
Observe that $J$ is the disjoint union of the
$I_k$'s with $i_k>0$ and the $I_k^T$'s with $i_k<0$.
Therefore we have $\chi_J=(-1)^d \chi_{I_1}\chi_{I_2}\cdots\chi_{I_{n-1}}$
and the assertion follows.
\end{proof}

Let us give an example.

\begin{example} \label{Primer-1}
Let $n>1$ and let $\sig\in S_n$ be the identity. We set $i_1=-1$ and
$i_k=k-1$ for $1<k<n$. The $i_k$'s are distinct and the condition $|i_k|\in\bZ_k$
is satisfied for all $k$. In this case we obtain the pattern
$$ J=\{(i,j):1\le i<j-1\le n-1\} \cup \{(i,1):1<i\le n\}. $$
Only $i_1$ is negative and so $d=n-1$, and $|i_k|=\sig(k)=k$ is valid
only for $k=1$.
By the theorem we have $\la \chi_J,\chi_n \ra=(-1)^{n-1}n.$
This proves the case $k=1$ of Conjecture \ref{Hess}.
\end{example}

Recall that $\NW_n$ is not $n$-universal for $n\ge3$
(see Example \ref{dij+1}). However, if we
modify this pattern to make it strict by replacing its diagonal
positions $(i,i)$ with $(i,n+1-i)$, we can show that the new pattern
$$ \Pi_n=\{(i,j):i+j\le n, i\ne j\} \cup \{(i,n-i+1):2i\le n\} $$
is nonsingular and, consequently, universal.
This is our second analogue of Schur's theorem.
The zero entries required by $\Pi_n$ for
$n=2,3,4,5$ are exhibited below:
\begin{equation*}
\left[ \begin{array}{cc}
* & 0 \\ * & *  \end{array} \right],\quad
\left[ \begin{array}{ccc}
* & 0 & 0 \\ 0 & * & * \\ * & * & * \end{array} \right],\quad
\left[ \begin{array}{cccc}
* & 0 & 0 & 0 \\ 0 & * & 0 & * \\ 0 & * & * & * \\ * & * & * & *
\end{array} \right],\quad
\left[ \begin{array}{ccccc}
* & 0 & 0 & 0 & 0 \\ 0 & * & 0 & 0 & * \\ 0 & 0 & * & * & * \\
0 & * & * & * & * \\ * & * & * & * & *
\end{array} \right].
\end{equation*}

\begin{proposition} \label{Primer-2}
We have $\la \chi_{\Pi_n}, \chi_n \ra=n!!.$
\end{proposition}
\begin{proof}
This is in fact a special case of Theorem \ref{Fam-J}.
We take $\sig\in S_n$ to be the permutation $1,n,2,n-1,\ldots.$
Thus $\sig(2k-1)=k$ for $2k-1\le n$ and $\sig(2k)=n+1-k$ for $2k\le n$.
We set ${\bf i}=(1,-1,2,-2,\ldots).$ The $i_k$'s are distinct. As
$i_{2k-1}=-i_{2k}=k=\sig(2k-1)$, the condition
$|i_k|\in\sig(\bZ_k)$ is satisfied for all $k$'s.
With this $\sig$ and {\bf i} we have $\Pi_n=J(\sig,{\bf i})$.
The equality $|i_k|=\sig(k)$ holds iff $k$ is odd and the inequality
$i_k<0$ holds iff $k$ is even. Thus $d=\sum(n-2k)$, the sum being over
all positive integers $k$ such that $2k\le n$. Therefore $d$
is even for $n$ even and $d\equiv [n/2] \pmod{2}$ for $n$ odd.
One can easily verify that $\sgn(\sig)=(-1)^d$ in all cases.
Hence, we obtain the formula given in the proposition.
\end{proof}

In several cases we used \cite{NS} to identify various sequences that
we have encountered, such as the double factorial sequence $A006882$
in the above proposition.

Let $J=J(\sig,{\bf i})$ be the pattern (\ref{Obl-J}).
Note that $J^T=J(\sig,-{\bf i})$,
where $-{\bf i}=(-i_1,-i_2,\ldots,-i_{n-1})$.
If $\rho\in S_n$ it is easy to verify that
$\rho(J)=J(\rho\sig,{\bf j})$ where ${\bf j}=(j_1,j_2,\ldots,j_{n-1})$
with
$$ j_k=\begin{cases} \rho(i_k) & \text{if } i_k>0; \\
-\rho(-i_k) & \text{otherwise}.
\end{cases} $$
It follows that $\pJ_n$ is a union of equivalence classes.

\begin{problem} \label{oblici}
Determine the number of equivalence classes contained in $\pJ_n$.
\end{problem}

For $n=2,3,4,5$ the answers are $1,2,7,34$ respectively.

In the case when $\sig=\iota$, the sequence
${\bf i}=(i_1,i_2,\ldots,i_{n-1})$
is subject only to the conditions:
(a) the integers $i_k$ are pairwise distinct and
(b) $i_k\in\bZ_k$ for all $k\in\bZ_{n-1}$.
In particular $i_1=\pm1$. To simplify the
notation, in this case we set $J({\bf i})=J(\iota,{\bf i})$.
From the previous discussion it is clear that each equivalence class
contained in $\pJ_n$ has a representative of the form $J({\bf i})$
with $i_1=1$. However,  $J({\bf i})\approx J({\bf j})$ may hold
for two different sequences ${\bf i}$ and ${\bf j}$ with $i_1=j_1=1$.
For instance, for $n=3$ we have $J((1,2))\approx J((1,-2))$.

\section{The case $n=4$} \label{Cetiri}

In this section we fix $n=4$. Recall that the nonsingular
patterns are universal, and the defective ones are not.
In this section we shall exhibit a strict pattern which is singular
and universal (see Proposition \ref{Oblik 6} below).
There are $7$ exceptional equivalence classes in $\pP'_4$,
their representatives are listed in Table 2. The last column of the
table shows what is known about the universality of the pattern.

\begin{center}
\begin{tabular}{ccl}
&\\
\multicolumn{3}{c}
{\bf  Table 2: Exceptional classes for $n=4$} \\
&\\
No. & Representative pattern & Univ. \\
&\\ \hline &\\
1 & \{(1,2),(2,1),(1,3),(3,1),(2,3),(3,2)\} & No \\
2 & \{(1,2),(2,1),(1,3),(1,4),(2,4),(3,2)\} & No \\
3 & \{(1,2),(1,3),(1,4),(2,1),(3,4),(4,3)\} & Yes \\
4 & \{(1,2),(2,1),(1,3),(3,1),(2,4),(4,3)\} & ? \\
5 & \{(1,2),(2,1),(1,3),(2,3),(4,1),(4,2)\} & ? \\
6 & \{(1,2),(2,1),(1,4),(2,3),(3,1),(4,2)\} & ? \\
7 & \{(1,2),(2,1),(1,4),(3,1),(3,4),(4,3)\} & ? \\
\end{tabular}
\end{center}

\bigskip

We shall prove now that the first two patterns are not universal.

\begin{proposition} \label{Oblik 1}
The first pattern in Table 2 is not universal.
\end{proposition}
\begin{proof}
Denote this pattern by $I$. Let $A\in L(4)$ be the matrix whose entries
in positions $(1,2)$ and $(3,4)$ are 1 and all other entries are 0.
Note that $A$ has rank 2 and that $A^2=0$.
Assume that $A$ is unitarily similar to some $X=[x_{ij}]\in L_I$.
Thus $X$ has the form
$$ X=\left[ \begin{array}{cccc}
* & 0 & 0 & * \\ 0 & * & 0 & * \\ 0 & 0 & * & * \\ * & * & * & *
\end{array} \right]. $$
Since $X$ must be nilpotent of rank 2, it is clear that at least one
of $x_{14}$, $x_{24}$, $x_{34}$ is nonzero, and also at least one
of $x_{41}$, $x_{42}$, $x_{43}$ is nonzero.
We may assume that $x_{14}\ne0$.
Then $X^2=0$ implies that $x_{42}=x_{43}=0$, and so $x_{41}\ne0$.
As $x_{22}$ and $x_{33}$ are eigenvalues of $X$,
we must have $x_{22}=x_{33}=0$.
From $X^2=0$ we deduce that $x_{24}=x_{34}=0$.
Thus only the four corner entries of $X$ may be nonzero.
As $X$ is nilpotent of rank 2, we have a contradiction.
\end{proof}

\begin{proposition} \label{Oblik 3}
The second pattern in Table 2 is not universal.
\end{proposition}

\begin{proof}
Denote this pattern by $I$ and let
\begin{equation*}
A=\left[ \begin{array}{rrrr}
0&0&0&0 \\
0&-i&-i&0 \\
0&-1&1&1-i \\
1&2&0&i-1 \end{array} \right],
\end{equation*}
a nilpotent matrix of rank 3.
Assume that $AU=UX$ for some $X=[x_{ij}]\in L_I$ and some
$U=[u_{ij}]\in\Un(4)$. Thus $X$ has the form
$$ X=\left[ \begin{array}{cccc}
* & 0 & 0 & 0 \\ 0 & * & * & 0 \\ * & 0 & * & * \\ * & * & * & *
\end{array} \right]. $$
Let $u_k$ denote the $k$-th column of $U$. By equating the entries in
$X=U^*AU$, we obtain that $x_{ij}=u_i^*Au_j$ for all $i,j$.

Since $X$ is nilpotent, we must have $x_{11}=0$. The first row of $X$
is zero, and the other three rows must be linearly independent because
$X$ has rank 3. Since the first row of $UX=AU$ is zero,
we conclude that $u_{12}=u_{13}=u_{14}=0$.
As $U$ is unitary, we also have $u_{21}=u_{31}=u_{41}=0$.
Without any loss of generality, we may assume that $u_{11}=1$.

Since $u_2^*Au_1=x_{21}=0$, we have $u_{42}=0$.

Assume that $u_{32}=0$. Then $u_{23}=u_{24}=0$, and we may assume
that $u_{22}=1$. Since $u_3^*Au_2=x_{32}=0$, we obtain that
$-\bar{u}_{33}+2\bar{u}_{43}=0$.
Since $u_3 \perp u_4$, we must have $u_{44}=-2u_{34}\ne0$.
Now we obtain a contradiction: $0=x_{24}=u_2^*Au_4=-iu_{34}$.

Thus we must have $u_{32}\ne0$. Then
$u_2=u_{32}(0,\xi,1,0)$ with $\xi=u_{22}/u_{32}$.
If $u_{24}=0$, then $u_2 \perp u_4$ implies that $u_{34}=0$ and
the condition $x_{24}=0$ again gives a contradiction.
Consequently, we must have $u_{24}\ne0$ and so
$u_4=u_{24}(0,1,-\bar{\xi},\eta)$ for some $\eta\in\bC$. Then we have
\begin{eqnarray*}
Au_2 &=& u_{32}(0,-i(1+\xi),1-\xi,2\xi), \\
Au_4 &=& u_{24}(0,i(\bar{\xi}-1),-1-\bar{\xi}+(1-i)\eta,2+(i-1)\eta).
\end{eqnarray*}
As we must have $Au_4\perp u_2$, we obtain that
\begin{equation} \label{prva}
i\xi(1-\xi)-(1+\xi)+(1+i)\bar{\eta}=0.
\end{equation}
Since $U^*Au_2$ is the second column of $X$, $Au_2$ must be a linear
combination of $u_2$ and $u_4$. Therefore
\begin{equation*}
\left| \begin{array}{ccc}
\xi & -i(1+\xi) & 1 \\
1 & 1-\xi & -\bar{\xi} \\
0 & 2\xi & \eta \end{array} \right|=0,
\end{equation*}
i.e.,
\begin{equation} \label{druga}
\eta(\xi(1-\xi)+i(1+\xi))+2\xi(1+|\xi|^2)=0.
\end{equation}
From (\ref{prva}) and (\ref{druga}) we obtain that
\begin{eqnarray*}
&& \xi(1-\xi)+i(1+\xi)=(i-1)\bar{\eta}, \\
&& (1-i)|\eta|^2=2\xi(1+|\xi|^2).
\end{eqnarray*}
Hence, $\xi=\lambda(1-i)$ for some real $\lambda\ge0$.

It follows that
\begin{eqnarray*}
|\eta|^2 &=& 2\lambda(1+2\lambda^2), \\
(i-1)\bar{\eta} &=& i+2\lambda+2i\lambda^2, \\
2|\eta|^2 &=& 4\lambda^4+8\lambda^2+1, \\
0 &=& (1-2\lambda+2\lambda^2)^2,
\end{eqnarray*}
which is a contradiction.
\end{proof}

We now give the promised example of a pattern which
is universal and singular.

\begin{proposition} \label{Oblik 6}
The third pattern in Table 2 is universal.
\end{proposition}

\begin{proof} Let $A$ be any linear operator on $\bC^4$
of trace 0.
We have to construct an orthogonal basis $\{a_1,a_2,a_3,a_4\}$
such that, with respect to
this new basis, the matrix of $A$ belongs to $L_I$.

For $a_1$ we choose an eigenvector of $A^*$.
The case when $a_1$ is also
an eigenvector of $A$ is easy and we leave it to the reader.
We extend $\{a_1\}$ to an orthogonal basis $\{a_1,b_1,b_2,b_3\}$
such that $Aa_1=\lambda a_1+b_3$ for some $\lambda\in\bC$.
The matrix of $A$ with respect to this new basis has the form
\begin{equation*}
\left[ \begin{array}{c|ccc} \lambda&0&0&0 \\
\hline 0& & & \\ 0& &B& \\ 1& & & \end{array} \right].
\end{equation*}

In order to complete the proof, it suffices to show that there exist
nonzero vectors $x,y\in a_1^\perp$ such that
\begin{equation} \label{treca}
Ax\perp y,\quad x\perp Ay,\quad x\perp y,\quad b_3\in{\rm span}\{x,y\}.
\end{equation}

We shall now work with the $A$-invariant subspace $a_1^\perp$.
Let $b_{ij}$, $i,j\in\{1,2,3\}$ be the entries of the submatrix $B$.
We may assume that $b_{31}\bar{b}_{23}\ne b_{32}\bar{b}_{13}$ because
such matrices form a dense open subset of $M(3)$.
We shall write vectors in $a_1^\perp$ by using their coordinates with
respect to the basis $\{b_1,b_2,b_3\}$. We shall seek the vectors
$x$ and $y$ in the form
\[ x=(a+ib,c+id,1),\quad y=(a+ib,c+id,-a^2-b^2-c^2-d^2), \]
where $a,b,c,d\in\bR$. Observe that the last two conditions in
(\ref{treca}) are automatically satisfied. The first two conditions give
\begin{eqnarray}
\label{cetvrta} && \left( b_{11}(a+ib)+b_{12}(c+id)+b_{13} \right)
(a-ib)+ \\
\notag && \left( b_{21}(a+ib)+b_{22}(c+id)+b_{23} \right) (c-id)- \\
\notag && \left( b_{31}(a+ib)+b_{32}(c+id)+b_{33} \right)
(a^2+b^2+c^2+d^2) =0
\end{eqnarray}
and
\begin{eqnarray} \label{peta}
&& \left( b_{11}(a+ib)+b_{12}(c+id)-b_{13}(a^2+b^2+c^2+d^2) \right)
(a-ib)+ \\
\notag && \left( b_{21}(a+ib)+b_{22}(c+id)-b_{23}(a^2+b^2+c^2+d^2)
\right)(c-id)+ \\
\notag && \left( b_{31}(a+ib)+b_{32}(c+id)-b_{33}(a^2+b^2+c^2+d^2) \right) \cdot1=0.
\end{eqnarray}
By taking the difference and cancelling the factor $1+a^2+b^2+c^2+d^2$,
we obtain the linear equation
\begin{equation} \label{sesta}
b_{13}(a-ib)+b_{23}(c-id)-b_{31}(a+ib)-b_{32}(c+id)=0.
\end{equation}

Thus our problem is reduced to showing that the equations (\ref{cetvrta})
and (\ref{sesta}) have a real solution for the unknowns $a,b,c,d$.
We now set $b_{ij}=b_{ij}'+ib_{ij}''$ where $b_{ij}',b_{ij}''\in\bR$
and denote by $(S)$ the system of four equations obtained from
(\ref{cetvrta}) and (\ref{sesta}) by equating to zero their real and
imaginary parts. The first two of these four equations are not
homogeneous. By homogenizing these two equations we obtain the system
which we denote by $(S')$. Although we are interested in real solutions,
we shall now consider all complex solutions of $(S')$ in the complex
projective 4-space. By B\'{e}zout's theorem, there are 9 solutions in the
generic case (counting multiplicities). We are going to show that
exactly two of these solutions lie on the hyperplane at infinity.
Consequently, the system $(S)$ has exactly 7 solutions (counting
multiplicities). Since the non-real solutions come in complex conjugate
pairs, at least one of them has to be real.
Clearly, this will complete the proof.

To find the solutions in the hyperplane at infinity (a complex
projective 3-space), we have to solve yet another homogeneous
system, $(S'')$, which is obtained from $(S)$ by omitting the terms
of degree less than 3 in the first two equations and retaining the
last two (linear) equations. The two new cubic equations factorize
as follows:
\begin{eqnarray*}
&& (a^2+b^2+c^2+d^2)(b_{31}'a-b_{31}''b+b_{32}'c-b_{32}''d)=0, \\
&& (a^2+b^2+c^2+d^2)(b_{31}''a+b_{31}'b+b_{32}''c+b_{32}'d)=0. \\
\end{eqnarray*}
If we assume that $a^2+b^2+c^2+d^2\ne0$, then we obtain a system of
four linear equations. The condition
$b_{31}\bar{b}_{23}\ne b_{32}\bar{b}_{13}$ is just saying that the
determinant of this system of linear equations is not 0. Thus,
this system has only the trivial solution. Hence the
solutions of $(S'')$ are just the solutions of the
system of two linear equations of $(S)$ and the equation
$a^2+b^2+c^2+d^2=0$. But a line intersects a quadric in exactly
two points and we are done.
\end{proof}

For the remaining four exceptional classes the universality
remains undecided.

\begin{problem} \label{Obl-4}
Decide which of the last four patterns in Table 2 are universal.
\end{problem}

\section{Counting reducing flags} \label{RedFlag}

Recall that the homogeneous space
$\pF_n=\Un(n)/T_n$ is known as the {\em flag variety}. It is a real
smooth manifold of dimension $n(n-1)$. The points of this
manifold can be interpreted in several ways. By the above definition,
the points are cosets $gT_n$, $g\in\Un(n)$. They can be viewed also as
complete flags
$$ 0=V_0\subset V_1\subset V_2\subset\cdots\subset V_n=\bC^n, $$
i.e., the increasing sequence of complex subspaces $V_k$ with
$\dim V_k=k$. We shall follow this practise and refer to the points
of $\pF_n$ as {\em flags}.
Yet another often used interpretation is to consider
the points of $\pF_n$ as ordered $n$-tuples $ (W_1,W_2,\ldots,W_n) $
of 1-dimensional complex subspaces of $\bC^n$ which are orthogonal
to each other under the standard inner product.

Let $I\in\pP'_n$ and $A\in L(n)$.
Assume that the orbit $\pO_A$ meets $L_I=L_I(n)$ and let
$\pX_A=\pO_A\cap L_I$. The maximal torus $T_n$ acts on $\pX_A$
(on the right) by conjugation, i.e., $(X,t)\to t^{-1}Xt$ where
$t\in T_n$ and $X\in\pX_A$.
Let us also introduce the set
$$ \pU_A=\{g\in\Un(n):g^{-1}Ag\in L_I\}, $$
on which $T_n$ acts by right multiplication.
If $g\in\pU_A$ then $g^{-1}Ag\in L_I$ and we say that the
flag $gT_n$ {\em reduces} $A$ to $L_I$. Thus the set $\pU_A/T_n$
can be identified with the set of all reducing flags of the
matrix $A$.

The map
\begin{equation} \label{ekv-jed}
\theta:\pU_A\to\pX_A,\quad \theta(g)=g^{-1}Ag,
\end{equation}
is surjective and $T_n$-equivariant. Our main objective in this
section is to prove that the induced map
\begin{equation} \label{jed-jed}
\hat{\theta}:\pU_A/T_n\to\pX_A/T_n
\end{equation}
is bijective, i.e., that we have a natural bijective correspondence
between the reducing flags (for $A$) and the $T_n$-orbits
in $\pX_A$.

We need three lemmas. The first one is valid for any connected
compact Lie group $G$. For any $g\in G$ let $Z_g$ denote the identity
component of the centralizer of $g$ in $G$.

\begin{lemma} \label{Centre}
Let $G$ be a connected compact Lie group and $H$ a connected closed
subgroup of maximal rank. For any $g\in G$, $g$ belongs to the center
of $Z_g$. If $g\in H$ then $H$ contains the center of $Z_g$.
\end{lemma}
\begin{proof}
Let $T$ be a maximal torus of $G$ such that $g\in T$.
As $T\subseteq Z_g$, we have $g\in Z_g$ and the first assertion
follows. If $g\in H$, we may assume that $T$ is chosen so that
$T\subseteq H$. Then $T$ is a maximal torus of $Z_g$, and so $T$
must contain the center of $Z_g$. As $T\subseteq H$, we are done.
\end{proof}

The next two lemmas deal with the case $G=\Un(n)$.

\begin{lemma} \label {Connect-1}
If $H_1$ and $H_2$ are connected closed subgroups of $\Un(n)$
of rank $n$, then $H_1\cap H_2$ is connected.
\end{lemma}
\begin{proof}
In view of the above lemma, it suffices to show that the center of
$Z_g$ is connected for all $g\in\Un(n)$. To prove this,
we may assume that $g$ is a diagonal matrix. It follows that
$$ Z_g \cong \Un(n_1)\times\Un(n_2)\times\cdots\times\Un(n_k), \quad
n_1+\cdots+n_k=n. $$
Hence the center of $Z_g$ is a torus.
\end{proof}

\begin{lemma} \label {Connect-2}
The centralizer of any $A\in M(n)$ in $\Un(n)$ is connected.
\end{lemma}
\begin{proof}
Let $A=A_1+iA_2$ where $A_1$ and $A_2$ are hermitian matrices.
Since $A_k$ is unitarily diagonalizable, its centralizer $H_k$ in
$\Un(n)$ is a closed connected subgroup of rank $n$. Hence,
the centralizer $H_1\cap H_2$ of $A$ in $\Un(n)$ is connected
by Lemma \ref{Connect-1}.
\end{proof}

We can now prove the desired result. Recall that $A\in L(n)$ is
$I$-generic if $\pO_A$ and $L_I$ intersect transversally.

\begin{theorem} \label{Corresp}
Let $I\in\pP'_n$ and let $A\in L(n)$ be $I$-generic. Then the map
$\hat{\theta}$ defined by (\ref{ekv-jed}) and (\ref{jed-jed}) is
bijective. Consequently, $N_{A,I}$ is the number of flags which reduce
$A$ to $L_I$.
\end{theorem}
\begin{proof}
Since $\theta$ is surjective, so is $\hat{\theta}$. In order to prove
that $\hat{\theta}$ is injective, it suffices to show that if
$g_1,g_2\in\pU_A$ are such that $g_1^{-1}Ag_1=g_2^{-1}Ag_2$, we denote
this matrix by $B$, then the element $h=g_1^{-1}g_2$ belongs to $T_n$.
If $\u(n)\subseteq M(n)$ is the space of skew-hermitian matrices,
the transversality hypothesis implies that
$$ \{X\in\u(n):[X,B]\in L_I\} $$
is the space $\t_n$ of the diagonal skew-hermitian matrices. Hence, the Lie
algebra of the centralizer of $B$ in $\Un(n)$ is contained in $\t_n$.
By the above lemma this centralizer is connected, and so must be
contained in $T_n$. As $h$ commutes with $B$, we have $h\in T_n$.
\end{proof}

\begin{remark} \label{IntN}
It follows from the theorem that the number $N_{A,I}$ coincides
with the number denoted by $N(A,M_I(n,\bC))$ in \cite{AD}.
\end{remark}

\section{The cyclic pattern in the case $n=3$} \label{CikObl}

In this section we consider only the case $n=3$ and the cyclic pattern
\[ I=\{(1,3),(2,1),(3,2)\}\in\pP'_3. \]
By Theorem \ref{TeorAD}, $I$ is universal.
For $A\in L(3)$ let $\pO_A=\{UAU^{-1}:U\in\Un(3)\}$ be its orbit.
Recall that $A$ (or $\pO_A$) is said to be $I$-generic
if $\pO_A$ and $L_I=L_I(3)$ intersect transversally. Since $I$ is fixed,
we shall drop the prefix ``$I$-'' and say just that $A$ is generic.
We denote by $\Theta$ the set of nongeneric matrices in $L(3)$.
Clearly $\Theta$ is a closed $\Un(n)$-invariant subset.

We shall study here the set $\Theta$, and the intersections
$\pX_A=\pO_A \cap L_I$ for generic $A\in L(3)$.
In particular, we are interested in the possible values
of the number, $N_A=N_{A,I}$, of $T_3$-orbits contained in $\pX_A$.

First we consider the intersection $\Theta\cap L_I$.
This set contains all points $A\in L_I$ such that $\pO_A$
and $L_I$ are not transversal at $A$.
An arbitrary matrix $A\in L_I$ can be written as
\begin{equation} \label{Mat-A}
A=\left[ \begin{array}{ccc} u & z & 0 \\ 0 & v & x \\ y & 0 & w
\end{array} \right], \quad w=-u-v.
\end{equation}

Define the homogeneous polynomial $P_1:L_I\to\bR$ of degree 6 in the
real and imaginary parts of the complex variables $x,y,z,u,v$ by:
\begin{eqnarray} \label{Pol-P1}
P_1(A) &=&
|(v-w)x^2|^2+|(w-u)y^2|^2+|(u-v)z^2|^2 \\ \notag
&& +\left( |(v-w)x|^2+|yz|^2 \right)\cdot \left( |v|^2+|w|^2-5|u|^2 \right) \\
\notag
&& +\left( |(w-u)y|^2+|zx|^2 \right)\cdot \left( |w|^2+|u|^2-5|v|^2 \right) \\
\notag
&& +\left( |(u-v)z|^2+|xy|^2 \right)\cdot \left( |u|^2+|v|^2-5|w|^2 \right) \\
\notag
&& +|(u-v)(v-w)(w-u)|^2.
\end{eqnarray}

\begin{proposition} \label{Hip-pov}
For $A\in L_I$ as above, $\pO_A$ and $L_I$ intersect
non-transversally at $A$ iff $A$ lies on the real hypersurface
$\Gamma_1\subset L_I$ defined by the equation $P_1=0$.
\end{proposition}

\begin{proof}
The tangent space to $\pO_A$ at the point $A$ consists of all matrices
$[A,X]=AX-XA$ with $X^*=-X$ and $\tr(X)=0$. If $X$ is diagonal, then
$[A,X]\in L_I$. Let $S$ be the subspace of skew-Hermitian matrices
with all diagonal entries 0. We see immediately that $\pO_A$ and
$L_I$ are transversal at $A$ iff $L_I+[A,S]=L(3)$.
A routine computation shows that the latter condition is
equivalent to the vanishing of a certain determinant, a homogeneous
polynomial of degree 6. It is not hard to verify that this polynomial
(unique up to a scalar factor) is $P_1$.
\end{proof}

\begin{corollary} \label{Gama-1}
We have
\begin{eqnarray}
\label{G-Teta} \Gamma_1 \subseteq \Theta \cap L_I, \\
\label{T-Gama} \Theta = \cup_{A\in\Gamma_1} \pO_A.
\end{eqnarray}
\end{corollary}
\begin{proof}
The inclusion (\ref{G-Teta}) is obvious.
If $B\in\Theta$ then there exists $A\in\pX_B$ such that
$\pO_B$ and $L_I$ intersect non-transversally at $A$. Hence $A\in\Gamma_1$
and $B\in\pO_A$. Thus (\ref{T-Gama}) is valid.
\end{proof}

Our next objective is to construct the (unique) irreducible real hypersurface
$\Gamma$ in $L(3)$ containing the set $\Theta$. It is given by an equation
$P=0$, where $P:L(3)\to\bR$ is a homogeneous polynomial of
degree 24 in 16 variables, the real and imaginary parts of the entries
of $X\in L(3)$ except the last entry.
The polynomial $P$ is invariant under the action of the direct product
$\Un(1)\times\SU(3)$, where $\Un(1)$ acts by multiplication with scalars of
unit modulus and $\SU(3)$ acts by conjugation.
It can be expressed as a polynomial in the invariants $i_1,i_2,\ldots,i_{16}$
listed in Appendix A.
As this expression has 203 terms, we shall give it separately in Appendix B.

We warn the reader that $P$ is rather large.
Denote by $P_I$ its restriction to the subspace $L_I$.
We run out of memory if we try to evaluate $P_I$ at an
arbitrary matrix in $L_I$ and expand it as a polynomial
in the 10 real variables (the real and imaginary parts of the complex
variables $x,y,z,u,v$). However, when we set the imaginary parts of $y$
and $z$ to 0, then we can expand $P_I$ and obtain 130571 terms. If we additionally
set the imaginary part of $x$ to 0, then the number of terms goes down to 50583.

Recall that the hypersurface $\Gamma\subset L(3)$ is defined by
the equation $P=0$.

\begin{proposition} \label{Hip-Pov}
The restriction $P_I$ admits the factorization:
\begin{equation} \label{Faktor}
P_I=P_1^2 P_2,
\end{equation}
where $P_1$ is defined by (\ref{Pol-P1}) and $P_2$ is a homogeneous polynomial
of degree 12 with integer coefficients. Thus
$\Gamma \cap L_I=\Gamma_1 \cup \Gamma_2$, where $\Gamma_2$
is the hypersurface defined by $P_2=0$. We also have
$\Theta\subseteq\Gamma$.
\end{proposition}

\begin{proof}
The first assertion can be verified by using a computer and suitable software,
e.g., {\sc Maple} \cite{MGH}.
The assertion that $\Gamma \cap L_I=\Gamma_1 \cup \Gamma_2$
is now obvious. Finally, the assertion $\Theta\subseteq\Gamma$ follows from
Corollary \ref{Gama-1} and the fact that $\pO_A\subseteq\Gamma$ for
$A\in\Gamma$.
\end{proof}

We prove next that $P_1$, $P_2$ and $P$ are irreducible.
In fact we have a stronger result.

\begin{proposition} \label{AbsIred}
The real polynomials $P_1$, $P_2$ and $P$ defined above are absolutely
irreducible (i.e., irreducible over $\bC$).
\end{proposition}

\begin{proof}
The absolute irreducibility of $P_1$ can be proved by using the
``absfact.lib'' library in {\sc Singular} \cite{GPSO5}.
To make this computation easier,
it suffices to check that after setting $y=2$, $z=1$,
and setting the imaginary parts of $x$ and $v$ to 0 and 1, respectively,
the resulting polynomial (having 69 terms) still has degree 6
and is absolutely irreducible.

The same method works for $P_2$. In this case we set $u=1$, $y=2$, $z=1$,
and we also set the imaginary parts of $x$ and $v$ to 0 and 1, respectively.
We obtain an absolutely irreducible polynomial of degree 12
(having 47 terms).

It is much harder to prove that $P$ is also absolutely irreducible.
(We were not successful in using the same method as above.)
Assume that $P$ has a nontrivial factorization $P=QR$ over $\bC$.
We can assume that $Q$ is irreducible.
Since $\Un(3)$ is connected and $P$ is $\Un(3)$-invariant, both
$Q$ and $R$ must be invariant. By restricting these polynomials to
$L_I$ and by using (\ref{Faktor}), we obtain
that $P_I=Q_I R_I=P_1^2 P_2.$ It follows that $Q_I$
is equal to $P_1$ or $P_2$ (up to a scalar factor).
We are going to show that this leads to a contradiction.

One can easily verify that, for
$$ A= \left[ \begin {array}{ccc} 3+3\,i&5&0\\\noalign{\medskip}0&3-3\,i&5
\\\noalign{\medskip}5&0&-6\end {array} \right]\in L_I, $$
we have $P_1(A)=0$ and $P_2(A)=-22675690800$, and so
$A\in\Gamma_1\setminus\Gamma_2.$

On the other hand, the matrix
$$ B= \left[ \begin {array}{ccc} -1&\frac12\,\sqrt {222+6\,\sqrt {69}}&0
\\\noalign{\medskip}0&\frac12\,(1-\sqrt {69})&0\\\noalign{\medskip}\frac12\,
\sqrt {222-6\,\sqrt {69}}&0&\frac12(1+\sqrt {69})\end {array} \right]\in L_I $$
satisfies $P_1(B)=89424$ and $P_2(B)=0$, and so
$B\in\Gamma_2\setminus\Gamma_1.$

We obtain a contradiction by showing that $B\in\pO_A$.
An explicit unitary matrix $X$ satisfying $AX=XB$ is given in Appendix C.
\end{proof}

We note that, in the above proof, $A$ is a regular point of $\Gamma_1$,
while $B$ is singular on $\Gamma_2$.

The subspace $L_I$ is $Z$-invariant, for the cyclic matrix
$$ Z=\left[ \begin{array}{ccc} 0 & 0 & 1 \\ 1 & 0 & 0 \\ 0 & 1 & 0
\end{array} \right]. $$
Indeed, if $A\in L_I$ is given by (\ref{Mat-A}) then
$$ ZAZ^{-1}=\left[ \begin{array}{ccc} w & y & 0 \\ 0 & u & z \\ x & 0 & v
\end{array} \right]\in L_I. $$

Assume now that $A\in L_I$ and that $P(A)\ne0$. We claim that the two points
$A,ZAZ^{-1}\in\pX_A$ are not $T_3$-conjugate. Otherwise
$ZAZ^{-1}=DAD^{-1}$ for some $D\in T_3$. This implies that $u=v=w$, and
$u+v+w=0$ forces that $u=v=w=0$. Consequently, $P_1(A)=0$ which
contradicts our assumption. We conclude that the three points of the
$Z$-orbit $\{A,ZAZ^{-1},Z^{-1}AZ\}$ belong to three different
$T_3$-orbits in $\pX_A$. Since $P_1$ is $Z$-invariant, we have
$$ P_1(A)=P_1(ZAZ^{-1})=P_1(Z^{-1}AZ). $$
It follows that $N_A$ is divisible by 3. As $N_A=N(A,M_I(3))$
by the remark \ref{IntN} and as we know that in this case
$N(A,M_I(3))$ is even, we conclude that $N_A$ is divisible by 6.
Numerical computations indicate that $N_A$ is either 6 or 18
and that $\pX_A/T_3$ can be split into 6-tuples such that
$P_1$ is constant on the union of the $T_3$-orbits
belonging to the same 6-tuple. We have no explanation for this
phenomenon.

Let us look at an example. The matrices
$$ A= \left[ \begin{array}{ccc} 1+i & 0 & 0 \\ 0 & -1 & 0 \\
1 & 0 & -i \end{array} \right] \quad \text{and} \quad
B= \left[ \begin{array}{ccc} -i & 0 & 0 \\ 0 & -1 & 0 \\
1 & 0 & 1+i \end{array} \right] $$
belong to $L_I$. As $P_1(A)=P_1(B)=45$ and $P_2(A)=P_2(B)=0$,
we have $A,B\in\Gamma_2\setminus\Gamma_1.$
One can easily verify that $A$ and $B$ are unitarily similar and
are regular points of $\Gamma_2$.
Hence $\pX_A$ contains the six $T_3$-orbits with representatives
$Z^kAZ^{-k},Z^kBZ^{-k}$, $k=0,1,2$. We propose $A$ and $B$ as
examples of generic points which belong to $\Gamma$, which
would imply that $\Theta\ne\Gamma$.

\section{Appendix A: Generators of $\Un(1)\times\SU(3)$-invariants} \label{UxSU}

Consider the representation of the direct product $\Un(1)\times\SU(3)$
on $L(3)$, where $\Un(1)$ acts by multiplication with scalars of
unit modulus and $\SU(3)$ acts by conjugation.
The algebra of real polynomial invariants for this action is
a subalgebra of the corresponding algebra for the conjugation action
of $\Un(3)$ on $L(3)$.
A minimal set of homogeneous generators of the first algebra
consists of 16 invariants $i_1,i_2,\ldots,i_{16}$ listed below.
This fact is neither proved nor used in this paper.
However, these generators are used in Sections \ref{UnivObl} and
\ref{CikObl} to construct some other invariants that we need.
The 16 generators are
\begin{eqnarray*}
&& i_1=\tr(XY),\, i_2=\tr(X^2 Y^2),\, i_3=\frac14 | \tr(X^2) |^2, \\
&& i_4=\tr(XYX^2 Y^2),\, i_5=\frac19| \tr(X^3) |^2,\, i_6=| \tr(X^2 Y) |^2, \\
&& i_7=\frac16\, \Re \left( \tr(Y^2)( 3\tr^2(X^2 Y)+\tr(X^3)\tr(XY^2) ) \right), \\
&& i_8=\frac16\, \Re \left( \tr(Y^2)( 3\tr^2(X^2 Y)-\tr(X^3)\tr(XY^2) ) \right), \\
&& i_9=\frac12\, \Im \left( \tr(Y^2)\tr^2(X^2 Y) \right), \\
&& i_{10}=\frac16\, \Im \left( \tr(X^2)\tr^2(X^2 Y)\tr(Y^3) \right), \\
&& i_{11}=\frac1{12}\, \Re( \tr^2(X^2)\tr(Y^3)\tr(XY^2) ), \\
&& i_{12}=\frac1{12}\, \Im( \tr^2(X^2)\tr(Y^3)\tr(XY^2) ), \\
&& i_{13}=\frac1{72}\, \Re( \tr^2(X^3)\tr^3(Y^2) ), \\
&& i_{14}=\frac1{72}\, \Im( \tr^2(X^3)\tr^3(Y^2) ), \\
&& i_{15}=\frac13\, \Im \left( \tr^3(X^2 Y)\tr(Y^3) \right), \\
&& i_{16}=\frac1{144}\, \Re( \tr^4(X^2)\tr^2(Y^3)\tr^2(XY^2) ),
\end{eqnarray*}
where $X\in L(3)$ is an arbitrary matrix, $Y=X^*$ is its adjoint, and $\tr$
is the matrix trace function.
The symbols $\Re$ and $\Im$ stand for ``real part'' and ``imaginary part'',
respectively.

\section{Appendix B: The polynomial $P$ } \label{Polinom-P}

Here we construct the polynomial $P$ used in Section \ref{CikObl}.
It will be given as a polynomial in the invariants $i_k$ listed in Appendix A.
For convenience, we collect $P$ with respect to the invariants $i_3$ and $i_6$
\begin{equation} \label{Pol-P}
P = \sum_{k,l} p_{kl}i_3^k i_6^l.
\end{equation}
The nonzero $p_{kl}$ are as follows:
\begin{eqnarray*}
p_{00} &=&
-6\, \left( 126\,{i_{{2}}}^{2}{i_{{1}}}^{4}-26\,i_{{2}}{i_{{1}}}^{6}+
336\,i_{{7}}i_{{2}}{i_{{1}}}^{2}-270\,{i_{{2}}}^{3}{i_{{1}}}^{2}+1536
\,{i_{{7}}}^{2} \right. \\
&& \left. +216\,{i_{{2}}}^{4}-11\,i_{{7}}{i_{{1}}}^{4}+2\,{i_{{1}
}}^{8}-1152\,i_{{7}}{i_{{2}}}^{2} \right)  \left( i_{{8}}+i_{{7}}
 \right) \\
p_{01} &=&
18432\,i_{{11}}i_{{7}}+186\,i_{{8}}{i_{{1}}}^{5}+2016\,i_{{11}}{i_{{1}
}}^{2}i_{{2}}+2160\,{i_{{2}}}^{2}i_{{1}}i_{{7}}-66\,i_{{11}}{i_{{1}}}^{4} \\
&& +2160\,i_{{8}}i_{{1}}{i_{{2}}}^{2}-1278\,i_{{8}}{i_{{1}}}^{3}i_{{2}
}+186\,{i_{{1}}}^{5}i_{{7}}-6912\,i_{{11}}{i_{{2}}}^{2}+3456\,i_{{1}}{
i_{{7}}}^{2} \\
&& +3456\,i_{{8}}i_{{1}}i_{{7}}-1278\,{i_{{1}}}^{3}i_{{2}}i_{{7}} \\
p_{02} &=&
3888\,i_{{8}}i_{{2}}+297\,{i_{{1}}}^{2}{i_{{2}}}^{2}-4608\,i_{{13}}-
324\,{i_{{2}}}^{3}-90\,{i_{{1}}}^{4}i_{{2}}+3888\,i_{{2}}i_{{7}} \\
&& +9\,{i_{{1}}}^{6}-1242\,{i_{{1}}}^{2}i_{{7}}-1242\,i_{{8}}{i_{{1}}}^{2}-3456
\,i_{{11}}i_{{1}} \\
p_{03} &=& 18\,i_{{1}} \left( -7\,{i_{{1}}}^{2}+27\,i_{{2}} \right) \\
p_{04} &=& 729 \\
p_{10} &=& 6912\,i_{{7}}i_{{4}}i_{{1}}i_{{2}}+2016\,i_{{2}}{i_{{1}}}^{2}i_{{13}}-
33\,i_{{4}}i_{{2}}{i_{{1}}}^{5}+1008\,i_{{4}}{i_{{2}}}^{2}{i_{{1}}}^{3} \\
&& -3456\,i_{{4}}i_{{1}}{i_{{2}}}^{3}-1008\,{i_{{4}}}^{2}i_{{2}}{i_{{1}}
}^{2}+31104\,{i_{{5}}}^{2}{i_{{2}}}^{2}+33\,i_{{5}}{i_{{1}}}^{7} \\
&& +297\,{i_{{5}}}^{2}{i_{{1}}}^{4}+9792\,{i_{{2}}}^{3}i_{{8}}-2304\,{i_{{4}}}^
{2}i_{{8}}-251\,{i_{{1}}}^{6}i_{{8}}-20736\,{i_{{5}}}^{2}i_{{8}} \\
&& -3024\,i_{{5}}i_{{4}}i_{{2}}{i_{{1}}}^{2}+25920\,i_{{7}}{i_{{2}}}^{3}
-55296\,{i_{{7}}}^{2}i_{{2}}-6912\,i_{{7}}{i_{{4}}}^{2} \\
&& +17760\,{i_{{7}}}^{2}
{i_{{1}}}^{2}+3456\,{i_{{4}}}^{2}{i_{{2}}}^{2}+9504\,i_{{5}}{i_{{2}}}^
{2}{i_{{1}}}^{3}-6912\,i_{{7}}i_{{5}}{i_{{1}}}^{3} \\
&& -66\,{i_{{1}}}^{4}i_{{13}}+4608\,i_{{16}}-24012\,i_{{7}}{i_{{2}}}^{2}{i_{{1}}}^{2}
-317\,i_{{7}}{i_{{1}}}^{6}-62208\,i_{{7}}{i_{{5}}}^{2} \\
&& +2304\,i_{{4}}i_{{2}}i_{
{1}}i_{{8}}-1206\,i_{{5}}i_{{2}}{i_{{1}}}^{5}+3012\,i_{{2}}{i_{{1}}}^{
4}i_{{8}}+5226\,i_{{7}}i_{{2}}{i_{{1}}}^{4} \\
&& -20736\,i_{{7}}i_{{5}}i_{{4
}}+8544\,i_{{7}}{i_{{1}}}^{2}i_{{8}}-6912\,{i_{{2}}}^{2}i_{{13}}+9216
\,i_{{7}}i_{{13}}+66\,{i_{{1}}}^{5}i_{{11}} \\
&& -9072\,{i_{{5}}}^{2}i_{{2}}
{i_{{1}}}^{2}-2304\,i_{{5}}{i_{{1}}}^{3}i_{{8}}-27648\,i_{{7}}i_{{2}}i
_{{8}}-6912\,i_{{5}}i_{{4}}i_{{8}} \\
&& -11052\,{i_{{2}}}^{2}{i_{{1}}}^{2}i_
{{8}}+13824\,i_{{5}}i_{{2}}i_{{1}}i_{{8}}+10368\,i_{{5}}i_{{4}}{i_{{2}
}}^{2}+1632\,{i_{{2}}}^{2}{i_{{1}}}^{6} \\
&& -5151\,{i_{{2}}}^{3}{i_{{1}}}^{
4}+7200\,{i_{{2}}}^{4}{i_{{1}}}^{2}-256\,i_{{2}}{i_{{1}}}^{8}+16\,{i_{
{1}}}^{10}-1728\,{i_{{2}}}^{5} \\
&& -2016\,i_{{2}}{i_{{1}}}^{3}i_{{11}}+
41472\,i_{{7}}i_{{5}}i_{{1}}i_{{2}}-18432\,i_{{7}}i_{{1}}i_{{11}}+6912
\,{i_{{2}}}^{2}i_{{1}}i_{{11}} \\
&& -20736\,i_{{5}}i_{{1}}{i_{{2}}}^{3}+99\,
i_{{5}}i_{{4}}{i_{{1}}}^{4}+33\,{i_{{4}}}^{2}{i_{{1}}}^{4} \\
p_{11} &=&
-4320\,i_{{5}}i_{{2}}{i_{{1}}}^{2}+6\,{i_{{1}}}^{3}i_{{8}}+6912\,i_{{7
}}i_{{4}}+1728\,{i_{{4}}}^{2}i_{{1}}-3450\,i_{{7}}{i_{{1}}}^{3} \\
&& -5088\,{i_{{1}}}^{2}i_{{11}}+14688\,i_{{1}}{i_{{2}}}^{3}+2721\,i_{{2}}{i_{{1}
}}^{5}+2304\,i_{{4}}i_{{8}}-4320\,i_{{2}}i_{{1}}i_{{8}} \\
&& +15552\,{i_{{5}
}}^{2}i_{{1}}+1152\,i_{{1}}i_{{13}}+11520\,i_{{5}}i_{{8}}-3456\,i_{{4}
}{i_{{2}}}^{2}+1440\,i_{{7}}i_{{1}}i_{{2}} \\
&& -20736\,i_{{5}}{i_{{2}}}^{2}
-272\,{i_{{1}}}^{7}+1530\,i_{{5}}{i_{{1}}}^{4}-33\,i_{{4}}{i_{{1}}}^{4
}+32256\,i_{{2}}i_{{11}} \\
&& +5184\,i_{{5}}i_{{4}}i_{{1}}-9792\,{i_{{2}}}^{
2}{i_{{1}}}^{3}-720\,i_{{4}}i_{{2}}{i_{{1}}}^{2}+39168\,i_{{7}}i_{{5}} \\
p_{12} &=&
-1728\,i_{{4}}i_{{1}}+2130\,{i_{{1}}}^{4}-10170\,i_{{2}}{i_{{1}}}^{2}+
15228\,{i_{{2}}}^{2}+1296\,i_{{8}} \\
&& +1296\,i_{{7}}-10368\,i_{{5}}i_{{1}} \\
p_{13} &=& -3726\,i_{{1}}
\end{eqnarray*}
\begin{eqnarray*}
p_{20} &=&
-4272\,i_{{4}}i_{{2}}{i_{{1}}}^{3}+16128\,i_{{4}}i_{{1}}{i_{{2}}}^{2}-
16128\,{i_{{4}}}^{2}i_{{2}}+12816\,i_{{5}}i_{{4}}{i_{{1}}}^{2} \\
&& -48384\,i_{{5}}i_{{4}}i_{{2}}+4272\,{i_{{4}}}^{2}{i_{{1}}}^{2}-27360\,{i_{{2}}
}^{2}i_{{8}}+23808\,i_{{7}}i_{{8}} \\
&& +80844\,i_{{7}}i_{{2}}{i_{{1}}}^{2}+
69888\,{i_{{7}}}^{2}-47808\,i_{{5}}i_{{2}}{i_{{1}}}^{3}+117504\,i_{{5}
}i_{{1}}{i_{{2}}}^{2} \\
&& -39168\,i_{{7}}i_{{5}}i_{{1}}+19500\,i_{{2}}{i_{{
1}}}^{2}i_{{8}}-8544\,{i_{{1}}}^{2}i_{{13}}+32256\,i_{{2}}i_{{13}} \\
&& +8544\,{i_{{1}}}^{3}i_{{11}}+384\,{i_{{1}}}^{8}+20160\,{i_{{2}}}^{4}-
128736\,i_{{7}}{i_{{2}}}^{2}-11607\,i_{{7}}{i_{{1}}}^{4} \\
&& +4470\,i_{{5}}{i_{{1}}}^{5}-145152\,{i_{{5}}}^{2}i_{{2}}
+38448\,{i_{{5}}}^{2}{i_{{1}}}^{2}-2799\,{i_{{1}}}^{4}i_{{8}} \\
&& -11520\,i_{{5}}i_{{1}}i_{{8}}-50016\,{i_{{2}}}^{3}{i_{{1}}}^{2}
-5279\,i_{{2}}{i_{{1}}}^{6}+26283\,{i_{{2}}}^{2}{i_{{1}}}^{4} \\
&& -32256\,i_{{2}}i_{{1}}i_{{11}} \\
p_{21} &=&
-22224\,i_{{7}}i_{{1}}-4272\,i_{{4}}{i_{{1}}}^{2}-37632\,i_{{11}}+
16128\,i_{{4}}i_{{2}}+96768\,i_{{5}}i_{{2}} \\
&& -2112\,{i_{{1}}}^{5}+18192\,i_{{2}}{i_{{1}}}^{3}-15264\,i_{{5}}{i_{{1}}}^{2}
-8400\,i_{{1}}i_{{8}}-40608\,i_{{1}}{i_{{2}}}^{2} \\
p_{22} &=& -10476\,i_{{2}}+10401\,{i_{{1}}}^{2} \\
p_{30} &=& 44448\,i_{{5}}{i_{{1}}}^{3}+31296\,i_{{8}}i_{{2}}
+4415\,{i_{{1}}}^{6}-76308\,{i_{{1}}}^{2}i_{{7}}-18816\,i_{{4}}i_{{1}}i_{{2}} \\
&& +169344\,{i_{{5}}}^{2}+56448\,i_{{5}}i_{{4}}
+128496\,{i_{{1}}}^{2}{i_{{2}}}^{2}+37632\,i_{{11}}i_{{1}} \\
&& -209664\,i_{{5}}i_{{1}}i_{{2}}+18816\,{i_{{4}}}^{2}-9108\,i_{{8}}{i_{{1}}}^{2}
+236352\,i_{{2}}i_{{7}}  \\
&& -37632\,i_{{13}}-90624\,{i_{{2}}}^{3}-43032\,{i_{{1}}}^{4}i_{{2}} \\
p_{31} &=& -117504\,i_{{5}}+480\,i_{{2}}i_{{1}}+3312\,{i_{{1}}}^{3}
-18816\,i_{{4}} \\
p_{32} &=& -5196 \\
p_{40} &=& -146816\,i_{{2}}{i_{{1}}}^{2}-10384\,i_{{8}}
+112896\,i_{{5}}i_{{1}}+23948\,{i_{{1}}}^{4} \\
&& +195840\,{i_{{2}}}^{2}-142480\,i_{{7}} \\
p_{41} &=& 33632\,i_{{1}} \\
p_{50} &=& -202048\,i_{{2}}+65232\,{i_{{1}}}^{2} \\
p_{60} &=& 78400
\end{eqnarray*}

\section{Appendix C: The unitary matrix $X$}

The entries of the unitary matrix $X=[x_{ij}]$, used in the proof of
Proposition \ref{AbsIred}, are given by
\begin{eqnarray*}
x_{11} &=& \frac { 2 \left( \sqrt {69}-7 \right)
-i \left( 3+\sqrt {69} \right) }{ 6\, \sqrt {37-\sqrt {69}}}, \\
x_{12} &=& \frac {1}{12\,\sqrt {13}}(1+3i)(i\sqrt {6}-\sqrt {46}), \\
x_{13} &=& \frac {1}{12}\,(\sqrt {46}+\sqrt {6}), \\
x_{21} &=& \frac {(3-i)\sqrt {69} +34+27i}{15\,\sqrt {37-\sqrt {69}} }, \\
x_{22} &=& \frac {4}{15\,\sqrt {13}} ( \sqrt {46}-\sqrt {6} )
+\frac {i}{60\,\sqrt {13}}\,(23\,\sqrt {6}-3\,\sqrt {46}), \\
x_{23} &=& \frac {1}{60}\,(3-i)(3\,\sqrt {6}-i\,\sqrt {46}), \\
x_{31} &=& \frac {2(2-\sqrt{69})-i(3+\sqrt{69})} {6\,\sqrt{37-\sqrt{69}} }, \\
x_{32} &=& \frac {4}{15\,\sqrt{13}}\,(\sqrt {6}+\sqrt {46})
+\frac{i}{60\,\sqrt{13}}\,(23\,\sqrt {6}+3\,\sqrt {46}), \\
x_{33} &=& \frac{1}{12}(\sqrt {46}-\sqrt {6}).
\end{eqnarray*}

\end{document}